\documentclass[12pt,a4paper]{amsart}

\usepackage{amssymb,color}
\usepackage[english]{babel}
\usepackage{verbatim,here}
\usepackage[T1]{fontenc}
\usepackage{floatflt,graphicx,graphics}
\usepackage{color,xcolor}
\usepackage{enumerate}
\usepackage{animate}
\usepackage{a4wide}
\usepackage{amsmath, amssymb}
\usepackage{amsthm}

\graphicspath{{
C://Users//vuorinen//Documents//manus09//nasser//ccciso//Latex//figures}}

\newcounter{minutes}
\divide\time by 60
\newcounter{hours}
\multiply\time by 60 \addtocounter{minutes}{-\time}
{\small
\address{Department of Mathematics, Statistics and Physics, Qatar University, P.O. Box 2713, Doha, Qatar}
\email{mms.nasser@qu.edu.qa}
\address{Department of Mathematics and Statistics, University of Turku, 20014 Turku, Finland}
\email{vuorinen@utu.fi}

}
\keywords{Condenser capacity, hyperbolic metric, isoperimetric problems, numerical computation, boundary integral equation}
\subjclass[2010]{30C85, 31A15, 65E10}

\dedicatory{}
\commby{}
\theoremstyle{plain}
\newtheorem{thm}[equation]{Theorem}
\newtheorem{cor}[equation]{Corollary}
\newtheorem{lem}[equation]{Lemma}

\theoremstyle{definition}

\newtheorem{rem}[equation]{Remark}
\theoremstyle{remark}

\newtheorem{nonsec}[equation]{}

\numberwithin{equation}{section}

\newcommand{\beq}{\begin{equation}}
\newcommand{\eeq}{\end{equation}}
\newcommand{\ben}{\begin{enumerate}}
\newcommand{\een}{\end{enumerate}}
\newcommand{\bequu}{\begin{eqnarray*}}
\newcommand{\eequu}{\end{eqnarray*}}
\newcommand{\bequ}{\begin{eqnarray}}
\newcommand{\eequ}{\end{eqnarray}}

\newcommand{\R}{\mathbb{R}}

\newcommand{\arth}{\,\textnormal{arth}}
\newcommand{\arsh}{\,\textnormal{arsh}}
\newcommand{\sh}{\,\textnormal{sh}}
\newcommand{\ch}{\,\textnormal{ch}}
\renewcommand{\th}{\,\textnormal{th}}

\newcommand{\C}{\mathbb{C}}

\newcommand{\D}{\mathbb{D}}

\newcommand{\mM}{\mathsf{M}}
\newcommand{\M}{\mathsf{M}}
\newcommand{\G}{\Gamma}
\newcommand{\g}{\gamma}
\newcommand{\col}{\, : \,}

\usepackage{url}
\usepackage{color}

\renewcommand{\i}{\mathrm{i}}

\newcommand{\bs}{{\bf s}}
\newcommand{\capa}{\mathrm{cap}}

\renewcommand{\thefootnote}{\number_style{footnote}}
\begin{document}

\def\thefootnote{}

\title[]{Isoperimetric properties of condenser capacity}

\author[M.M.S. Nasser]{Mohamed M.S. Nasser}
\author[Matti Vuorinen]{Matti Vuorinen}

\date{}

\begin{abstract}
For compact subsets $E$ of the unit disk $ \mathbb{D}$ we study the capacity of the condenser
${\rm cap}( \mathbb{D},E)$ by means of set functionals defined in terms of
hyperbolic geometry. In particular, we study experimentally the case of a hyperbolic
triangle and arrive at the conjecture that of all triangles with the same hyperbolic area, the equilateral triangle has the least capacity. 
\end{abstract}

\maketitle

\footnotetext{\texttt{{\tiny File:~\jobname .tex, printed: \number\year-%
\number\month-\number\day, \thehours.\ifnum\theminutes<10{0}\fi\theminutes}}}
\makeatletter

\makeatother




\section{Introduction} \label{section1}
\setcounter{equation}{0}

Solutions to geometric extremal problems  often  exhibit symmetry---the extremal configurations are symmetric although the initial configuration is not. The classical example from antiquity is the isoperimetric problem which ask to find the planar domain with largest area, given its perimeter \cite{ban}. The solution is the circle. Here one studies the relationship between two {\it domain functionals}, the area of the domain and the perimeter of its boundary. 

The classical book of G. Polya and G. Szeg\"o \cite{pos} is a landmark of the study of extremal problems. The extremal problems they studied had geometric flavour and most of these problems have their roots in mathematical physics, but the authors called these isoperimetric problems because of their similarity to the classical geometric problem. 
One of their main topics was to investigate extremal problems of condenser capacity. 
The notion of a condenser has its roots in Physics and the mathematical study of capacity belongs to potential theory. Given a simply connected domain $G$ in the plane and a compact set $E$ in $G\,,$ the pair $(G,E)$ is called a {\it condenser} and its {\it capacity} is defined as
\[
{\capa} (G,E)={\rm inf}_u \int_G| \nabla u|^2 \, dm
\]
where the infimum is taken over all functions $u\,:\,G\to \R$ in $C^{\infty}_o(G)$  with $u(x) \ge 1 $ for all $x \in E\,.$ For a large class of sets $E$ it is known that the infimum is attained by a harmonic function $u:G \to \mathbb{R}$ which is a solution to the classical {\it  Dirichlet problem} for the Laplace equation
\[
\Delta u = 0\,, \quad u(x) =0\,,\; x \in \partial G\,,\; u(x) =1\,,\; x \in E\,.
\]

The capacity cannot usually be expressed as an analytic formula. 
The capacities of even the simplest geometric condensers, for example 
when $G$ is the unit disk and $E$ is a triangle or a square, seem to be unknown. 
Rather we could say that the cases when explicit formulas exist are exceptional.
This being the case, it is natural
to look for upper or lower bounds  or numerical approximations for the capacity. 
An important extremal property of the capacity is that it decreases under a geometric transformation called symmetrization \cite{bs,pos,sar}. After the symmetrization the condenser is transformed
onto another symmetric condenser and its capacity might be possible to estimate in terms of well-known special functions \cite{avv}.
Several types of symmetrizations are studied in \cite{pos,sar} and in the most recent literature \cite{bae,du,kes}.

In their study of isoperimetric problems and symmetrization, Polya and Szeg\"o used domain functionals to study extremal problems for condenser capacities.
The capacity is a conformal invariant \cite{ah,bae} and this fact has many
applications both to theory \cite{hkv,garmar} and to practice \cite{sl}.

A natural approach to study capacity would make use of conformal invariance. Euclidean geometry is invariant under similarity transformations but  not under conformal maps.
Thus it seems appropriate to use the conformally invariant hyperbolic geometry when studying capacity. We apply here numerical methods developed by the first author in a series of papers, see e.g., \cite{Nas-ETNA} and the references cited therein. This method, based on boundary integral equations, enables us to compute the capacity when $G$ is the unit disk $\mathbb{D}$ and the set $E$ is of very general type with piecewise smooth
boundary. Here we will study the case when $E$ is a hyperbolic polygon.

In the Euclidean geometry the sum of the angles of a triangle equals $\pi$ whereas in the 
hyperbolic geometry, the hyperbolic area of a triangle with angles  $\alpha, \beta,\gamma \,$
equals
\[
\pi -(\alpha+\beta+\gamma) \,.
\]
The two domain functionals of a hyperbolic triangle $T$, the hyperbolic area and its capacity
${\rm cap}(\mathbb{D}, T)$ are both conformally invariant and therefore we expect an explicit
formula also for the capacity. Surprisingly enough, we have not been able to find such a formula in the literature.
Some leading experts of geometric function theory we contacted also were not aware of such a formula.

In this paper we have made an effort to introduce all the basic facts from the hyperbolic geometry so as to make our paper as  self-contained as possible, using the relevant pages from \cite{be} as a source. After the preliminary material we provide a description of our computational method. Then we describe the algorithms for computing
the  capacities ${\rm cap} (\mathbb{D},E)\,$ of hyperbolic polygons $E\,$ and give our main results in the form of tables, experimental error analysis of computations, and graphics.

Our work and experiments lead to several conjectures including those about an isoarea property of the
capacity. For instance, our results  support the conjecture
 that among all hyperbolic triangles $T$ of a given area, the equilateral hyperbolic triangle $T_0$ has the least capacity,
\[
\capa (\mathbb{D},T) \ge \capa (\mathbb{D},T_0) \,.
\]

Finally, we remark that isoperimetric and isoarea problems for condenser capacities and other domain functionals have been analyzed in more general setting in the recent preprints \cite{bch, dmm}.

\section{Preliminary results}

In this section we summarize the few basic facts about the hyperbolic geometry of the unit disk
that we use in the sequel \cite{be}. This geometry is non-euclidean, the parallel axiom does not hold. 
The fundamental difference between the Euclidean geometry of $\mathbb{C}$ and the
hyperbolic geometry of  $\mathbb{D}$ is different notion of invariance: while the Euclidean
geometry is invariant with respect to translations and rotations, the hyperbolic
geometry is invariant under the groups of M\"obius automorphisms of  $\mathbb{D}\,.$
We follow the notation and terminology from \cite{be,hkv}. For instance, Euclidean disks are
denoted by
\[ B^2(x,t) = \{y \in \mathbb{C}: \, |x-y|<t\} \,.\]

\begin{nonsec}{\bf Hyperbolic geometry.} \label{hypgeo} \cite{garmar,be}
 For $x,y \in \mathbb{D}$ the {hyperbolic distance} $\rho_{\mathbb{D}}(x,y)$ is defined by
 \begin{equation} \label{myrho}
 \sh \frac{\rho_{\mathbb{D}}(x,y)}{2} = \frac{|x-y|}{\sqrt{(1-|x|^2)(1-|y|^2)}}\,.
 \end{equation}
The main property of the hyperbolic distance is the invariance under the M\"obius
 automorphisms of the unit disk $\mathbb{D}$ of the form
\[
T_a:  z \mapsto \frac{z-a}{1- \overline{a}z}\,.
\] 
These transformations preserve hyperbolic length and area.
In the metric space $(\mathbb{D}, \rho_{\mathbb{D}})$ one can build a non-euclidean geometry, where the parallel axiom does not hold. In this geometry, usually called the hyperbolic geometry of the Poincare disk, lines are circular arcs perpendicular to the boundary $\partial \mathbb{D} \,.$  Many results of Euclidean geometry and trigonometry have counterparts in the hyperbolic geometry \cite{be}.
 
 Let $G$ be a Jordan domain in the plane. One can define the hyperbolic metric on $G$ in terms of the  conformal Riemann mapping function $h: G \to \mathbb{D}= h(G)$ as follows:
 \[
 \rho_G(x,y) = \rho_{\mathbb{D}}(h(x),h(y))\,.
 \]
 This definition yields a well-defined metric, independent of the conformal mapping $h\,$
 \cite{be,kela}. In hyperbolic geometry the boundary $\partial G$ has the same role as the point of $\{\infty\}$ in Euclidean geometry: both are like ``horizon", we cannot see beyond it.
\end{nonsec} 

\begin{nonsec} {\bf Hyperbolic disks.} We use the notation
\[ B_\rho(x,M) = \{z \in \mathbb{D}: \rho(x,z)<M\}\]
for the hyperbolic disk centered at $x\in \mathbb{D}$ with radius $M>0\,.$
It is a basic fact that they are Euclidean disks with the center and radius given
by \cite[p.56, (4.20)]{hkv}
\begin{equation}\label{2.22.}
 \begin{cases}
        B_\rho(x,M)=B^2(y,r)\;,&\\
  \noalign{\vskip5pt}
      {\displaystyle y=\frac{x(1-t^2)}{1-|x|^2t^2}\;,\;\;
        r=\frac{(1-|x|^2)t}{1-|x|^2t^2}\;,\;\;t={\th}\, ( M/2)\;,}&
\end{cases}
\end{equation}
Note the special case $x= 0$,
\begin{equation} \label{hypDat0}
 B_\rho(0,M)=B^2(0,{\th}\, ( M/2)) \,.
\end{equation}
\end{nonsec} 

It turns out that
the hyperbolic geometry is more useful than the Euclidean geometry when studying the
inner geometry of domains in geometric function theory.

\begin{nonsec}{\bf Capacity of a ring domain.} {\rm 
A ring domain $D$ has two complementary components, compact sets $E$ and $F$ such that
$D= \mathbb{C}\setminus (E \cup F).$ It can be understood as a condenser
$( D \cup E, E)\,.$ In particular, 
the capacity of the annulus $\{ z \in \mathbb{C}:
a<|z|<b\}$ is given by \cite{avv}, \cite[(7.3)]{hkv}
\begin{equation} \label{anncap}
2\pi/\log(b/a) \,.
\end{equation}
Another ring domain with known capacity is the Gr\"otzsch ring or condenser 
$( \mathbb{D},[0,r]), 0<r<1\,.$ Its capacity can be expressed in terms of the complete
elliptic integral $\mathcal{K}(r)$ as follows \cite{avv}.
First define the decreasing homeomorphism $\mu: (0, 1] \rightarrow [0,\infty)$ by
\begin{equation*}\label{muDef}
\mu(r)=\frac{\pi}{2}\frac{\mathcal{K}(\sqrt{1-r^2})}{\mathcal{K}(r)}\,, \quad \mathcal{K}(r)=\int_{0}^{\pi/2} \frac{dt}{\sqrt{1-r^2\sin^2 t}} \,,
\end{equation*}
for $ r \in (0,1)\,, \, \mu(1)=0\,$.
Now the Gr\"otzsch capacity  can be expressed as follows \cite[p.122, (7.18)]{hkv}
\begin{equation} \label{GrCap}
2 \pi/ \mu(r)\,.
\end{equation}
}

\end{nonsec}

\begin{nonsec}{\bf Domain functionals and extremal problems.}
{\rm Numerical charateristics of geometric configurations are often studied in terms of domain
functionals. In this paper we study condensers and their capacities in terms of area and perimeter. The book of Polya and Szeg\"o \cite{pos} studies a large spectrum of these problems and many later researcher have continued their work. See the books of C. Bandle \cite{ban} and Kesavan
\cite{kes} for isoperimetric problems and A. Baernstein \cite{bae} and V.N. Dubinin\cite{du} for classical analysis and geometric function theory. In addition, the papers Sarvas \cite{sar},
Brock-Solynin \cite{bs}, and Betsakos \cite{bet} should be mentioned.

The condenser capacity is invariant under conformal mapping. Therefore it is a natural idea to express the domain functionals in terms of conformally invariant geometry. There are various results
for condenser capacity which reflect this invariance, but we have not seen a systematic study
based on these ideas.

We study condensers of the form $(\mathbb{D}, E)$ where $E$ is a compact set. In this
case we use the hyperbolic geometry to define domain functionals. Our initial point is to record
the relevant data \cite[p.132, Thm 7.2.2]{be} for the above two explicitly known cases, for the condensers $( \mathbb{D}, E_j)$
with $j=1,2$ where $E_1=\overline{B}_{\rho}(0,M)$ and $E_2=[0, {\th} (M/2)]\,.$ We  consider 
$E_2$ as a degenerate, extremely thin rectangle and therefore take its perimeter to be 
equal to $2M$ which is twice its hyperbolic diameter $M=\rho(0, {\th} (M/2))$.

Now consider a disk $E_1=\overline{B}_{\rho}(0,M_1)$ and a segment $E_2=[0, {\th} (M_2/2)]$ (we consider $E_2$ as a very thin rectangle) with the same hyperbolic
perimeter $c$. Then
\[
M_1 = \arsh\frac{c}{2\pi}, \quad M_2=\frac{c}{2}.
\]
Hence, by \eqref{anncap},
\[
f_1(c)\equiv\capa(\D,E_1)=\frac{2\pi}{\log(1/{\th}(M_1/2))}
=\frac{2\pi}{\log\left(1/{\th}\frac{{\arsh(c/(2\pi))}}{2}\right)}
=\frac{2\pi}{\log\left(\sqrt{1+\frac{4\pi^2}{c^2}}+\frac{2\pi}{c}\right)},
\]
and, by \eqref{GrCap},
\[
f_2(c)\equiv\capa(\D,E_2)=\frac{2\pi}{\mu({\th}(M_2/2))}
=\frac{2\pi}{\mu({\th}(c/4))}.
\]

Figure~\ref{myisopertst2} shows that for a fixed value of the hyperbolic perimeter $c>0\,,$ {\it  disks have larger capacity than very thin rectangles of the same hyperbolic
perimeter,} in other words $f_1(c) > f_2(c)$ for all $c>0\,.$ This conclusion  led us to
discover  the following,
apparently new, inequality for the special function $\mu$
\begin{equation}\label{newMuLow}
\frac{\pi}{2}>\frac{\mu(t) }{ \log\left(\sqrt{1+ u^2 } +u\right) }>1\,, \quad  u =\frac{\pi}{2 \arth\, t} \, \,,\,\, t\in(0,1)\,.
\end{equation}

 In the following sections we will study variations of this theme for
hyperbolic triangles and polygons.

\vspace{3mm}
\begin{table}[h!]
\caption{Two well-known capacities. For the perimeter, see \cite[p.132]{be}.}
\label{mytable2}
\begin{center}
\begin{tabular}{|r|p{7.5cm}|p{3cm}|}
\hline
Set & Capacity & Perimeter \\[2mm]
\hline 
{$E_1$ } &  $2\pi /\log(1/{\th}(M/2)) $&$2 \pi \,{\sh} M$  \\[2mm]
\hline
{$E_2$ } &  $ 2\pi /\mu({\th}(M/2))   $&$2M$ \\[2mm]
\hline
{$E_1$ } &  $f_1(c)\equiv
2\pi/\log(\sqrt{1+4\pi^2/c^2}+2\pi/c) $&$c$  \\[2mm]
\hline
{$E_2$ } &  $f_2(c)\equiv2\pi/\mu({\th}(c/4)) $&$c$ \\[2mm]
\hline
\end{tabular}
\end{center}
\end{table}

\begin{figure}[hbt] %
\centerline{
\scalebox{0.6}{\includegraphics[trim=0 0 0 0,clip]{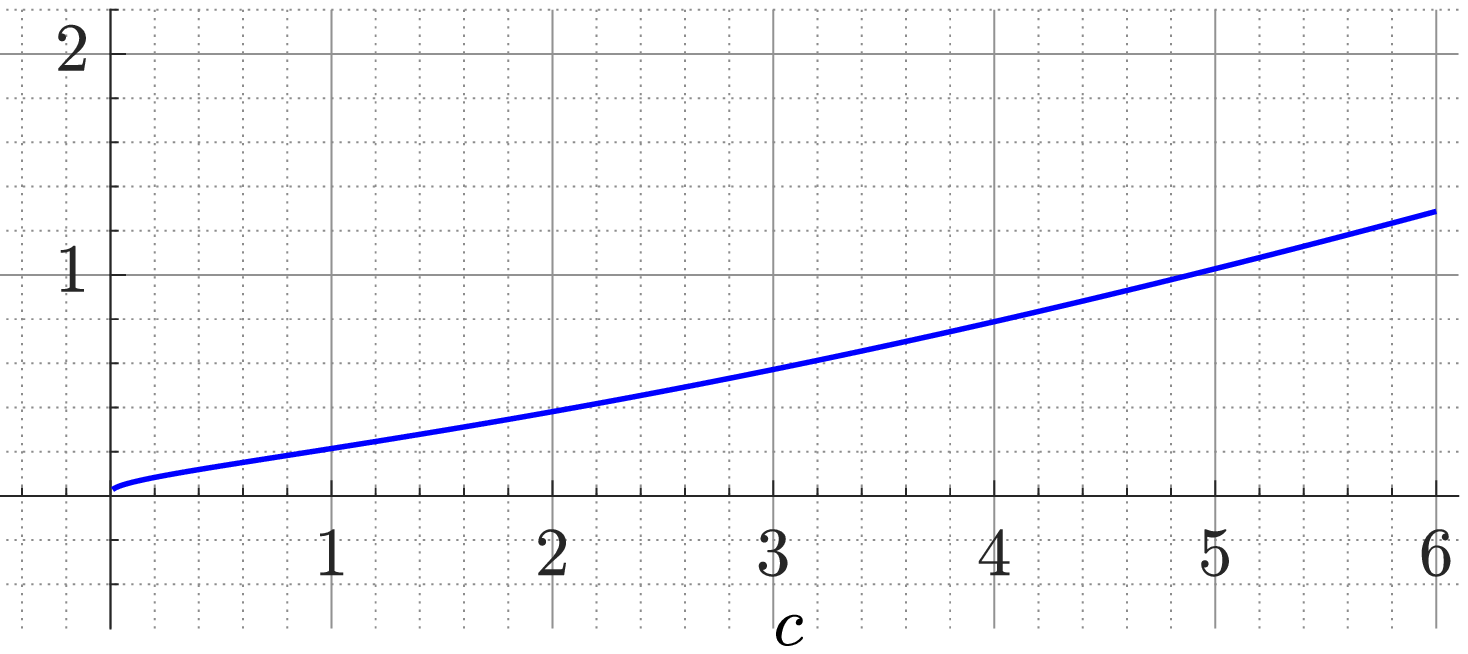}}
}
\caption{The difference of the functions $f_1(c)-f_2(c)$ defined in Table \ref{mytable2}}
\label{myisopertst2}
\end{figure}

 }
 \end{nonsec}

 \begin{nonsec}{\rm {\bf Modulus of a curve family.} \label{modulus}
For the reader's convenience we summarize some basic facts about the moduli
of curve families and their relation to capacities from the well-known sources
\cite{ah,du,garmar, hkv, LV}.
 Let $\G$ be a family of curves in $\R^n$. By $ {\mathcal A} (\G)$ we denote the family
of {\it admissible\/} functions, i.e.\ non--negative Borel--measurable
functions $ \rho \col \R^n\to \R\cup\{\infty\}$ such that
$$
   \int_{\g}  \rho\,ds\ge1
$$
for each locally rectifiable curve $\g$ in $\G$. For $p\ge 1$ the
$ p$--{\it modulus\/} of $\G$ is defined by
\begin{equation}\label{5.1.}
  \M_p(\G)=\inf_{ \rho \in {\mathcal A}  (\G)}\int_{\R^n}  \rho^p \, dm\;,
\end{equation}
where $m$ stands for the $n$--dimensional Lebesgue measure.
If $ {\mathcal A}  (\G)=\emptyset$, we set $\M_p(\G)=\infty$. The case $ {\mathcal A} (\G)=
\emptyset$ occurs only if there is a constant path in $\G$ because
otherwise the constant function $\infty$ is in $ {\mathcal A} (\G)$. Usually $p=n, n=2,$ and
we denote $\M_n(\G)$ also by $\M(\G)$ and call it the {\it modulus\/}
of $\G$. 
}
\end{nonsec}

\begin{lem}\label{le71}  \cite[7.1]{hkv} The $ p$--modulus $\M_p$ is an outer measure in the
space of all curve families in $\R^n$. That is,

{\rm(1)} $\M_p(\emptyset)=0\,$,                                     

{\rm(2)} $\G_1\subset \G_2$ implies $\M_p(\G_1)\le \M_p(\G_2)\,$,

{\rm(3)} $\M_p\Bigl({\displaystyle \bigcup_{i=1}^{\infty}}\,\G_i\Bigr)\le
       {\displaystyle {\sum_{i=1}^{\infty}}} \M_p(\G_i)\;.$
\end{lem}

\medskip
Let $\G_1$ and $\G_2$ be curve families in $\R^n$. We say that $\G_2$
is {\it minorized\/} by $\G_1$ and write $\G_2>\G_1$ if every $\g\in \G_2$
has a subcurve belonging to $\G_1$.
\medbreak

\begin{lem}\label{le72}  \cite[7.2]{hkv} $\G_1<\G_2$ implies $\M_p(\G_1)\ge \M_p(\G_2)$.
\end{lem}

The curve families  $\G_1,\G_2,\dots$ are called {\it separate\/}\index{separate curve families} if there
exist disjoint Borel sets $E_i$ in $\R^n$ such that if $\g\in \G_i$ is
locally rectifiable then $\int_{\g}\chi_ids=0$ where $\chi_i$ is the
characteristic function of $\R^n\setminus E_i$.

\begin{lem}\label{le73}  \cite[7.3]{hkv} If $\G_1,\G_2,\dots$ are separate and if $\G<\G_i$ for
all $i$, then
$$
\M_p(\G)\ge \sum \M_p(\G_i)\;.
$$
\end{lem}

The set of all curves joining two sets $E,F \subset G$ in $G$ is denoted by $\Delta(E,F;G)\,.$
The next result gives an alternative way to define the capacity of a condenser.

\begin{thm}\label{thm96}  \cite[9.6]{hkv} If $E=(A,C)$ is a bounded condenser in $\R^2$, then
$$
   \capa E=\M_2\bigl(\Delta(C,\partial A;A)\bigr)\;.
$$
\end{thm}

One of the fundamental properties of the modulus is its conformal invariance \cite{ah}, \cite{hkv}
and by Theorem \ref{thm96} we immediately see that the condenser capacity is a conformal invariant, too.

\nonsec{\rm {\bf Numerical computing of the capacity.} \label{numcap}
A numerical method for computing the capacity of condensers for doubly connected domains is presented in~\cite{nv}. The method is based on using the boundary integral equation with the generalized Neumann kernel~\cite{Nas-Siam1,Weg-Nas}. A fast numerical method for solving the integral equation is presented in~\cite{Nas-ETNA} which makes use of the Fast Multipole Method toolbox~\cite{Gre-Gim12}.

Let $G$ be a bounded simply connected domain and $E$ be a compact set in $G$ such that $D= G\setminus E$ is a doubly connected domain. Then, there exists a conformal mapping $w=f(z)$ which maps $D$ onto the annulus $\{w\,:\, q<|w|<1\}$ where $q$ is an undetermined real constant depending on $D$. The conformal mapping can be uniquely determined by assuming that $f(\alpha)>0$ for a given point $\alpha$ in $D$ (see Figure~\ref{fig:domain}). Since capacity is invariant under conformal mapping, the formula~\eqref{anncap} implies that
\[
\capa(G,E)=\frac{2\pi}{\log(1/q)}.
\]

\begin{figure}[hbt]
\centerline{\hfill
\scalebox{0.4}{\includegraphics[trim=0cm 0cm 0cm 0cm,clip]{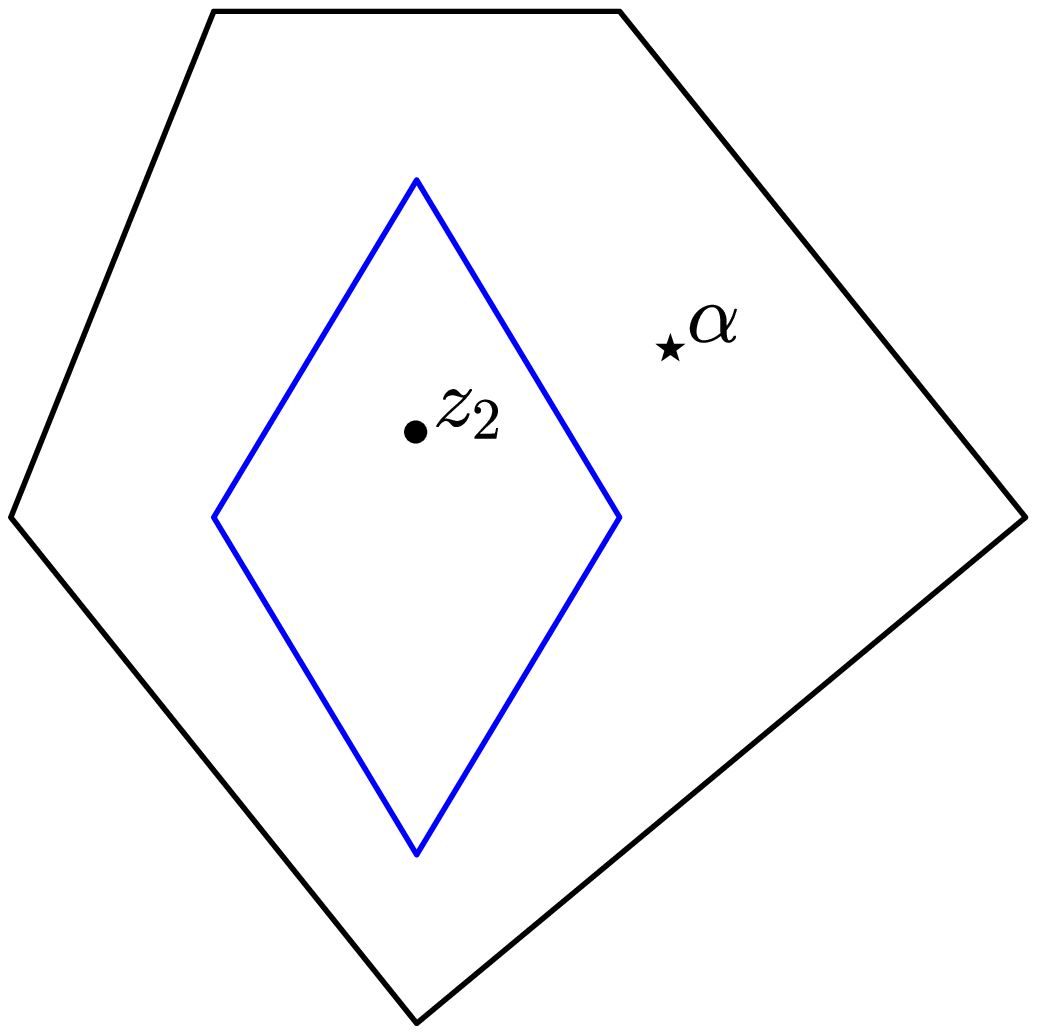}}\hfill
\scalebox{0.4}{\includegraphics[trim=0cm 0cm 0cm 0cm,clip]{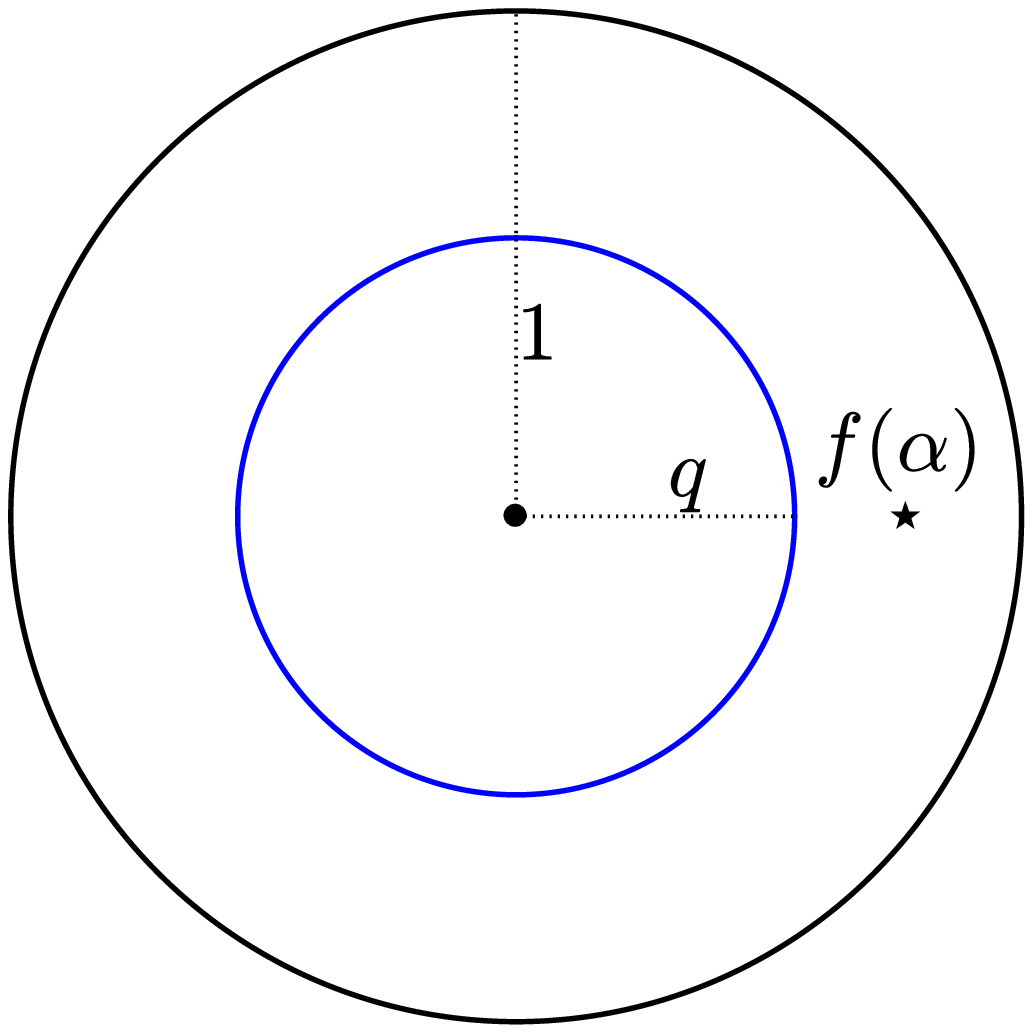}}\hfill}
\caption{The domain $D$ (left) and the annulus domain (right).}
\label{fig:domain}
\end{figure}

In this paper, we shall use the MATLAB function \verb|annq| from~\cite{nv} to compute $\capa(G,E)$ where the boundary components of $D= G\setminus E$ are assumed to be piecewise smooth Jordan curves. 
We denote the external boundary component of $D$ by $\Gamma_1$ and the inner boundary component by $\Gamma_2$. These boundary components are oriented such that, when we proceed along the boundary $\Gamma_1\cup\Gamma_2$, the domain $D$ is always on the left side. 
To use the function \verb|annq|, we parametrize each boundary component $\Gamma_j$ by a $2\pi$-periodic complex function $\eta_j(\delta_j(t))$, $t\in[0,2\pi]$, where $\delta_j\,:\,[0,2\pi]\to[0,2\pi]$ is a bijective strictly monotonically increasing function,
$j=1,2$. When $\Gamma_j$ is smooth, we choose $\delta_j(t)=1$. For piecewise smooth boundary component $\Gamma_j$, the function $\delta_j$ is chosen as described in~\cite[p.~697]{LSN17}. We define $n$ equidistant nodes $s_1, \ldots, s_n$ in the interval $[0,2\pi]$ by
\begin{equation}\label{eq:s_i}
s_k = (k-1) \frac{2 \pi}{n}, \quad k = 1, \ldots, n,
\end{equation}
where $n$ is an even integer. 
In MATLAB, we compute the vectors \texttt{et} and \texttt{etp} by
\begin{eqnarray*}
\texttt{et} &=& [\eta_1(\delta_1(\bs))\,,\,\eta_2(\delta_2(\bs))]\in\C^{2n}, \\
\texttt{etp} &=& [\eta_1'(\delta_1(\bs))\delta_1'(\bs)\,,\,\eta_2'(\delta_2(\bs))\delta_2'(\bs)]\in\C^{2n},  
\end{eqnarray*}
where $\bs=[s_1,\ldots, s_n]\in\R^n$. Then the capacity of the condenser $(G,E)$ is computed by calling
\begin{verbatim}
  [~,cap] = annq(et,etp,n,alpha,z2,'b'),
\end{verbatim}
where $z_2$ is  an auxiliary point in $E$, i.e., $z_2$ is chosen in the domain bounded by $\Gamma_2$ (See Figure~\ref{fig:domain}).
For more details, we refer the reader to~\cite{nv}.
The codes for all presented computations in this paper are available in the link \url{https://github.com/mmsnasser/iso}.


\section{The unit disk and a hyperbolic polygon}

In this section we compare  the capacities ${\rm cap}(\mathbb{D}, P)$ of hyperbolic polygons of equal hyperbolic area to the corresponding capacity when all the sides are of equal hyperbolic length. Our experiments suggest that in the latter case the capacity is minimal.

We reported this experimental result to A. Yu. Solynin, who has considered similar questions for a different notion of capacity \cite{so1,so2,soz}. In their significant paper \cite{soz} Solynin  and Zalgaller have proved a famous conjecture which expresses
a similar extremal property conjectured by Polya and Szeg\"o for some other capacity.
We are indebted to A. Solynin for these references and useful exchange of emails. However,
for the case of the capacity considered here,  at the present time, there is no analytic verification of our conjectured lower bound.

\nonsec{\bf Hyperbolic triangle.}
First, we consider condensers of the form $(\mathbb{D},T)$
where $T$ is a closed hyperbolic triangle with vertices
$s_1,s_2,s_3 \in \mathbb{D}\,.$ The sides of $T$ are subarcs of circles
orthogonal to $ \partial \mathbb{D}\,,$ each subarc joining two vertices.
Denote the the angles at vertices $s_1,s_2,s_3$ by $\sigma_1,\sigma_2,\sigma_3\,,$ respectively.
Then the hyperbolic area of $T$ is given by \cite[p. 150, Thm 7.13.1]{be}
\begin{equation} \label{Tharea}
\mbox{h-area}(T)=\pi-(\sigma_1+ \sigma_2+ \sigma_3)\,.
\end{equation}
It is a basic fact that $\mbox{h-area}(T)$ is invariant under M\"obius transformations
of $ \mathbb{D}\,$ onto itself. Also the capacity of the condenser $( \mathbb{D}, T)$
has the same invariance property.

\begin{nonsec}{\bf Open problem.}\label{HAopen} Given $s_1,s_2,s_3 \in  \mathbb{D}\,,$ find a formula for $ {\rm cap}( \mathbb{D}\,,T)\,.$
 \end{nonsec}

 Motivated by the simple fact that symmetry is often connected with extremal problems,
 we arrived at the following conjecture.

 \begin{nonsec}{\bf Conjecture.}\label{HAconj} 
Let $s_1,s_2,s_3 \in \mathbb{D}\,,$ and let
 $T$ be the hyperbolic triangle with vertices $s_1,s_2,s_3\,.$ If $T_0$ is an equilateral
 hyperbolic triangle with $\mbox{h-area}(T_0)=\mbox{h-area}(T)$, then
\begin{equation}\label{eq:Conjecture}
{\rm cap}( \mathbb{D}\,,T)   \ge {\rm cap}( \mathbb{D}\,,T_0)\,.
\end{equation}
 \end{nonsec}

 By M\"obius invariance we may without loss of generality normalize $T_0$ as follows.
 Write $\omega=(\sigma_1+ \sigma_2+ \sigma_3)/3\,.$ If $A,B,C$ are the lengths of the sides
 opposite to the angles $\sigma_1, \sigma_2, \sigma_3\,,$ resp., then by \cite[p.150, Ex. 7.12(2)]{be} the triangle is equilateral iff $\sigma_1= \sigma_2= \sigma_3=\omega$ and
 \begin{equation}\label{bep150}
 2 \ch(A/2) \sin(\omega/2) =1\,.
 \end{equation}
 We may assume also that $$ s_1=r\,,\, s_2= r e^{i\theta}\,,\, s_3=r e^{i2 \theta}\,,\, \theta = 2\pi/3\,.$$
 By \cite[p. 40]{be} we have
 $$ \sh^2 (A/2) = \left(\frac{r \sqrt{3}}{1-r^2}\right)^2 = \frac{1}{4 \sin^2(\omega/2)} -1$$
 and solving this for $r^2$ we obtain
 \begin{equation}\label{eq:hyp-tri-r}
 r^2 =\frac{2 \cos \omega - \sqrt{3} \sin \omega}{2 \cos \omega -1} \,.
 \end{equation}

 These observations show that given a hyperbolic triangle $T$ with angles $\sigma_1, \sigma_2, \sigma_3\,,$ there is a hyperbolic triangle $T_0$ with vertices
 $$ r\,,\,   r e^{i \theta}\,,\,  r e^{i 2\theta}\,,\, \theta = 2\pi/3 $$
 where $r^2$ is given by~\eqref{eq:hyp-tri-r} with 
 $\omega = (\sigma_1+ \sigma_2+ \sigma_3)/3$.
Both hyperbolic triangles $T$ and $T_0$ have the same hyperbolic area.

We compute $\capa(\D,T)$ numerically using the MATLAB function \verb|annq| with  $n=3\times10^{12}$ where the domain $D$ is the bounded doubly connected domain in the interior of the unit circle and in the exterior of the triangle (see Figure~\ref{fig:hyp_tri} where the auxiliary points $\alpha$ and $z_2$ in \verb|annq| are shown as the star and the dot, respectively). The values of $\capa(\D,T_0)$ are computed similarly.
The approximate values of the capacities $\capa(\D,T)$ and $\capa(\D,T_0)$ for several values of $s_1$, $s_2$, and $s_3$ are presented in Table~\ref{tab:hyp_tri}.
The presented numerical results validate the conjectural inequality~\eqref{eq:Conjecture}. Numerical experiments for several other values $s_1$, $s_2$, and $s_3$ (not presented here) also validate the conjectural inequality~\eqref{eq:Conjecture}.

\begin{figure}[hb]
\centerline{\hfill
\scalebox{0.4}{\includegraphics[trim=0cm 0.75cm 0cm 0.25cm,clip]{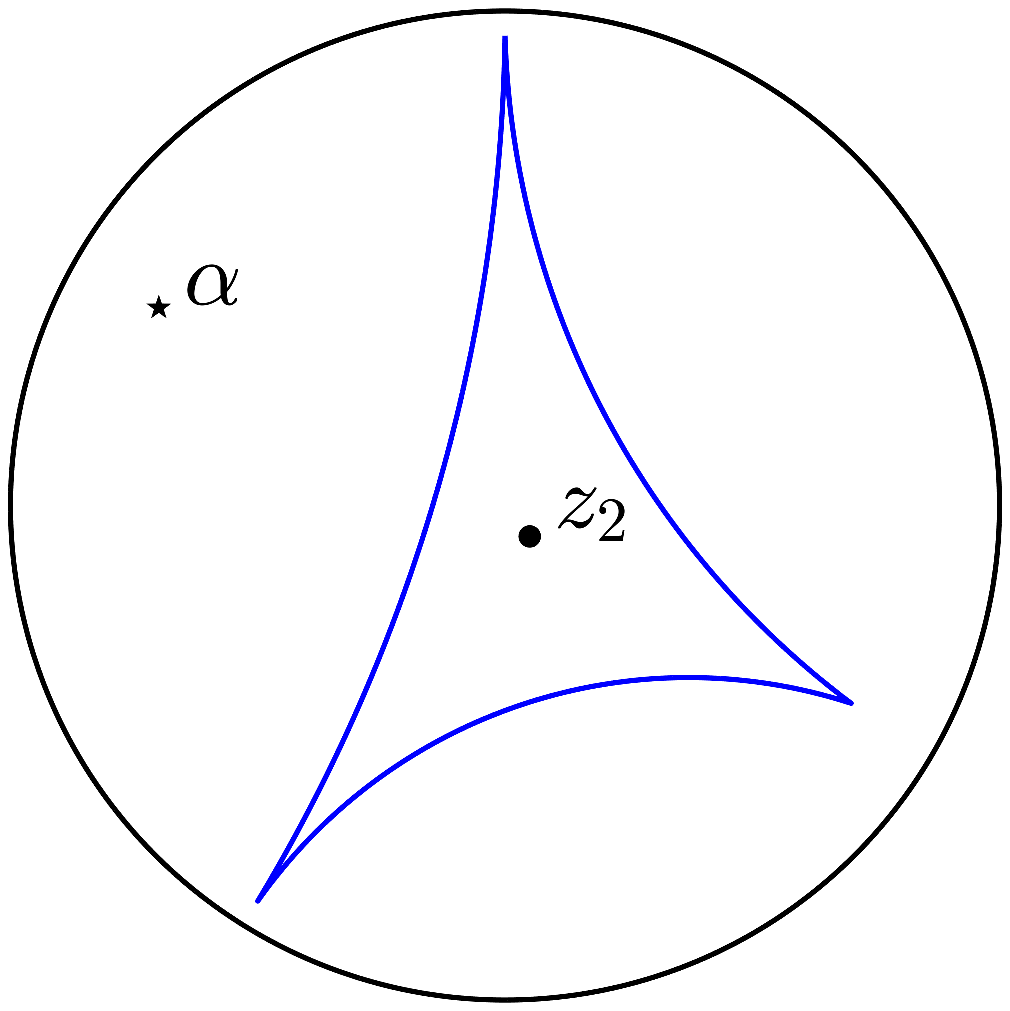}}\hfill
\scalebox{0.4}{\includegraphics[trim=0cm 0.75cm 0cm 0.25cm,clip]{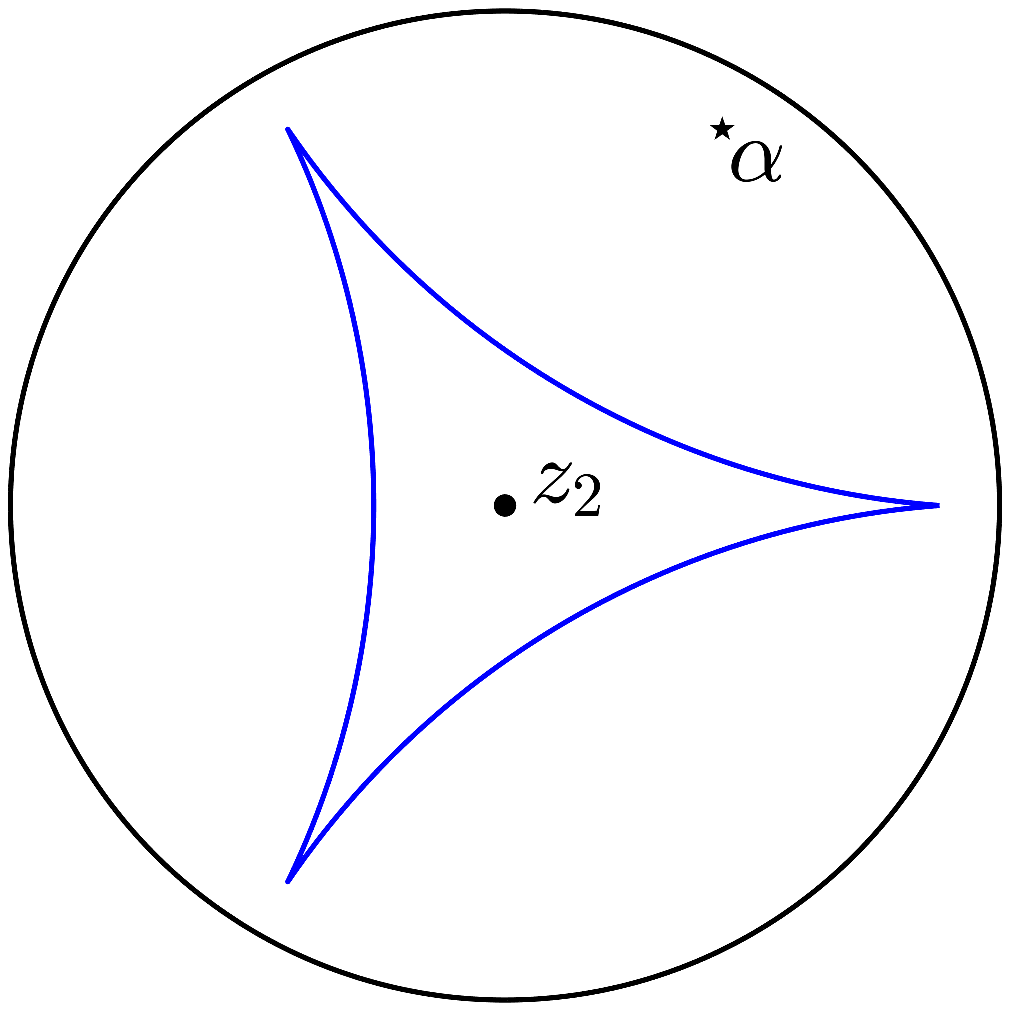}}\hfill}
\caption{The domain $D$ for computing $\capa(\D,T)$ (left) and the domain 
$D_0$ for computing $\capa(\D,T_0)$ (right) for $s_1=0.95\i$, $s_2=0.7-0.4\i$, $s_3=-0.5-0.8\i$.}
\label{fig:hyp_tri}
\end{figure}

\begin{table}[hbt]
\caption{The values of $\capa(\D,T)$ and $\capa(\D,T_0)$.}
\label{tab:hyp_tri}%
\begin{tabular}{l|l|l|l|l}\hline
$s_1$      & $s_2$            & $s_3$      & $\capa(\D,T)$       & $\capa(\D,T_0)$    \\ \hline
$0.6$    & $0.2-0.5\i$    &$-0.3-0.5\i$    & $5.61438997196548$  & $4.96507462804135$ \\
$0.9$    & $0.2-0.5\i$    &$-0.3-0.5\i$    & $7.57256635825877$  & $5.39880575287883$ \\
$0.3\i$  & $0.3-0.5\i$    &$-0.3-0.5\i$    & $5.63768713031744$  & $5.60191869448996$ \\ \hline
$0.5\i$  & $0.25-0.4\i$   &$-0.25-0.4\i$   & $5.52754816211627$  & $5.21348957109432$ \\
$0.9\i$  & $0.78-0.45\i$  &$-0.78-0.45\i$  & $13.2881689301735$  & $13.2881521954927$ \\
$0.95\i$ & $0.7-0.4\i$    &$-0.5-0.8\i$    & $13.92508317827$    & $12.4763956630121$ \\
\hline
$0.2\i$  & $0.17-0.1\i$   &$-0.17-0.1\i$   & $3.23750018859583$  & $3.23740547036233$ \\
$0.1\i$  & $0.087-0.05\i$ &$-0.087-0.05\i$ & $2.40145519669907$  & $2.40145213607884$ \\
$-0.1\i$ & $0.5-0.5\i$    &$-0.5-0.5\i$    & $5.98941024500545$  & $4.85509874205801$ \\
$-0.1\i$ & $0.7-0.5\i$    &$-0.7-0.5\i$    & $8.25251632029587$  & $4.89997376235771$ \\
\hline %
\end{tabular}
\end{table}

\nonsec{\bf Hyperbolic polygon with $m$ vertices.}
Second, we consider condensers of the form $(\mathbb{D},P)$
where $P$ is a closed hyperbolic polygon with $m$ vertices $\beta_1,\beta_2,\ldots,\beta_m\in\mathbb{D}$ such that $0\in P$. 

The hyperbolic distance between any two points $z,w\in\D$ can be computed by
\eqref{myrho}.
Thus, the perimeter of the hyperbolic polygon $P$ is
\[
L=\sum_{k=1}^{m}\rho_{\mathbb{D}}(\beta_k,\beta_{k+1})
\]
where $\beta_{m+1}=\beta_1$. Let $P_0$ be the hyperbolic polygon centered at $0$ and the hyperbolic length of all of its sides are equal to $L/m$. Then $P$ and $P_0$ have the same hyperbolic perimeter $L$. Assume that the vertices of the hyperbolic polygon $P_0$ are
\begin{equation}\label{eq:hyp-ver-vk}
v_k=r\,\exp\left(\frac{2\pi k\i}{m}\right), \quad k=0,1,2,\ldots,m.
\end{equation}
Define $v_{m+1}=v_1$, then the perimeter of the hyperbolic polygon $P_0$ is
\[
L=\sum_{k=1}^{m}\rho_{\mathbb{D}}(v_k,v_{k+1})
=\sum_{k=1}^{m}2\,\arsh
\left(\frac{|v_k-v_{k+1}|}{\sqrt{1-|v_k|^2}\sqrt{1-|v_{k+1}|^2}}\right),
\]
which, in view of~\eqref{eq:hyp-ver-vk}, can be written as
\[
L=\sum_{k=1}^{m}2\,\arsh\left(\frac{r|1-\exp\left(\frac{2\pi\i}{m}\right)|}{1-r^2}\right)
=2m\,\arsh\left(\frac{2r\sin\frac{\pi}{m}}{1-r^2}\right).
\]
Then, $r$ can be computed through
\[
r=\frac{-\sin\frac{\pi}{m}+\sqrt{\sin^2\frac{\pi}{m}+\sh^2\frac{L}{2m}}}{\sh\frac{L}{2m}}
\]

\begin{table}[hbt]
\caption{The values of $\capa(\D,P)$ and $\capa(\D,P_0)$.}
\label{tab:hyp_pol}%
\begin{tabular}{l|l|l|l}\hline
$m$  & $\beta_j,\; j=1,2,\ldots,m$ & $\capa(\D,P)$       & $\capa(\D,P_0)$ \\ \hline
$3$  & $0.6,0.1-0.8\i,-0.5+0.6\i$  & $9.0274303701827$   & $9.07270475215184$ \\[0.2cm]
\hline
$4$  & $0.601,-0.6\i,-0.599,0.6\i$ & $8.3279404581868$  & $8.32794231176445$ \\[0.2cm]
\hline
$5$  & $0.6,0.1-0.8\i,-0.5-0.5\i,-0.5+0.6\i,$ & $11.9589944965738$  & $12.0640771315217$ \\
     & $0.5+0.5\i$      &                     &                    \\[0.2cm]
\hline     
$6$  & $0.6,0.1-0.8\i,-0.5-0.5\i,-0.8,$ & $13.5302396750603$  & $13.6288953941389$ \\
     & $-0.5+0.6\i,0.5+0.5\i$ &                     &                    \\[0.2cm]
\hline
$7$  & $0.6,0.1-0.8\i,-0.5-0.5\i,-0.8,$ & $15.9302933204252$  & $16.0808062702908$ \\
     & $-0.5+0.6\i,0.9\i,0.5+0.5\i$    &                     &                    \\[0.2cm]
\hline
$8$  & $0.6,0.5-0.5\i,0.1-0.8\i,-0.5-0.5\i,-0.8,$  & $16.7814228075178$  & $16.9697317440523$ \\
     & $-0.5+0.6\i,0.9\i,0.5+0.5\i$ &                   &                    \\[0.2cm]
\hline
$12$ & $0.7+0.2\i,0.7-0.2\i,0.4-0.5\i,-0.8\i,$ & $20.8062404526496$  & $21.0023784573094$ \\
     & $-0.4-0.7\i,-0.7-0.4\i,-0.8,-0.7+0.3\i,$ &                     &                    \\
     & $-0.4+0.7\i,0.9\i,0.3+0.8\i,0.5+0.5\i$ &                     &         \\
\hline %
\end{tabular}
\end{table}

 \begin{nonsec}{\bf Conjecture.}\label{HAconj2} For the above two hyperbolic polygons $ P$ and $P_0$,
\begin{equation}\label{eq:Conjecture2}
{\rm cap}( \mathbb{D}\,,P)   \le {\rm cap}( \mathbb{D}\,,P_0)\,.
\end{equation}
 \end{nonsec}

The MATLAB function \verb|annq| with $n=m\times10^{12}$ is used to compute approximate values of $\capa(\D,P)$ and $\capa(\D,P_0)$ where the domain $D$ is the bounded doubly connected domain in the interior of the unit circle and in the exterior of the hyperbolic polygon (see Figure~\ref{fig:hyp_pol} where the auxiliary points $\alpha$ and $z_2$ in \verb|annq| are shown as the star and the dot, respectively). The computed approximate values for several values of $m$ and $\beta_1,\beta_2,\ldots,\beta_m$ are presented in Table~\ref{tab:hyp_pol}. These numerical results validate the inequality~\eqref{eq:Conjecture2}. 
Numerical experiments for several other values $m$ and $\beta_1,\beta_2,\ldots,\beta_m\in\mathbb{D}$ (not presented here) also validate the conjecture inequality~\eqref{eq:Conjecture2}.

Note that in our experiments we have used hyperbolic polygons starlike with respect to~$0$: in other words, each radius intersects the polygonal curve exactly at one point.

Finally, Figure~\ref{fig:hyp-pol-r} and Table~\ref{tab:hyp_pol_rm} present approximate  values of the capacity of the hyperbolic polygon $P_0$ for several values of $m$ and $r$. For $m=3$, the polygon $P_0$ is an equilateral hyperbolic triangle.

\begin{figure}[hbt] %
\centerline{
\hfill\scalebox{0.4}{\includegraphics[trim=0cm 0.75cm 0cm 0.25cm,clip]{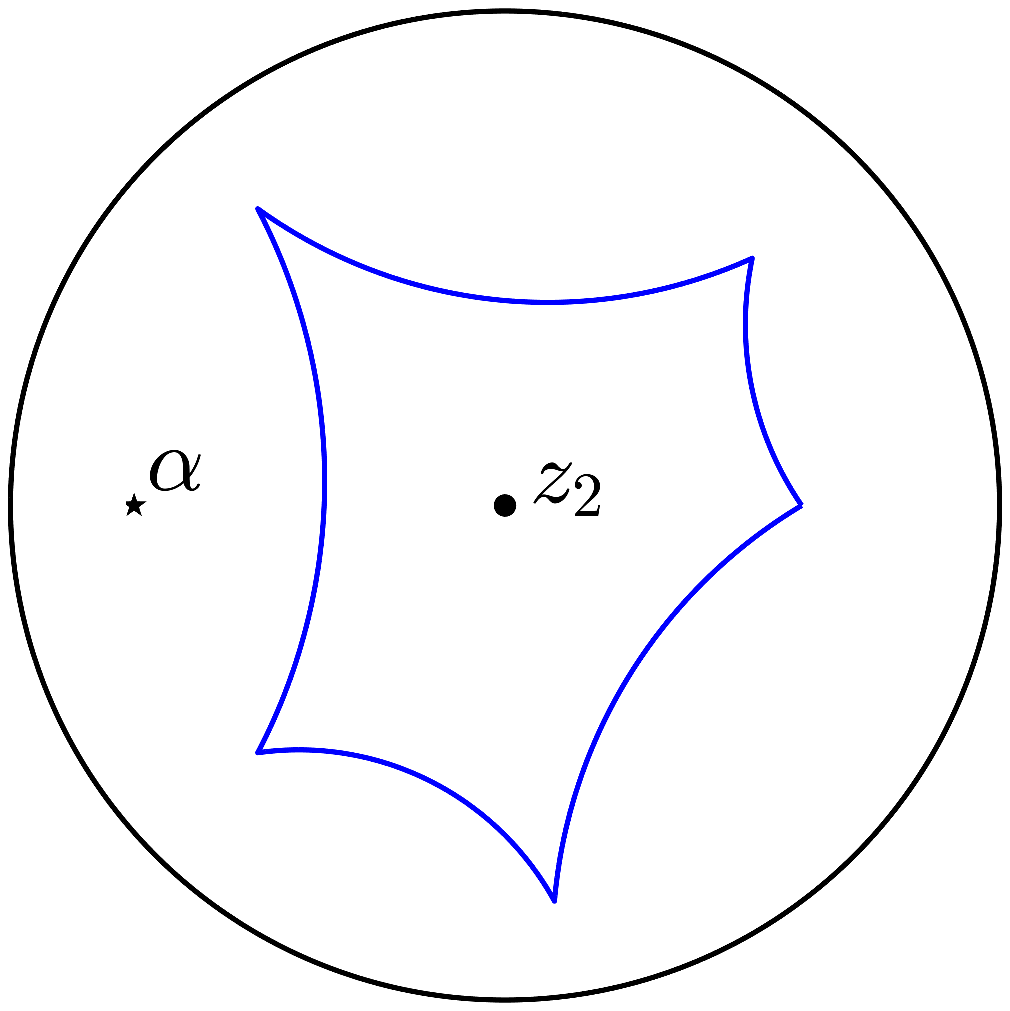}}
\hfill\scalebox{0.4}{\includegraphics[trim=0cm 0.75cm 0cm 0.25cm,clip]{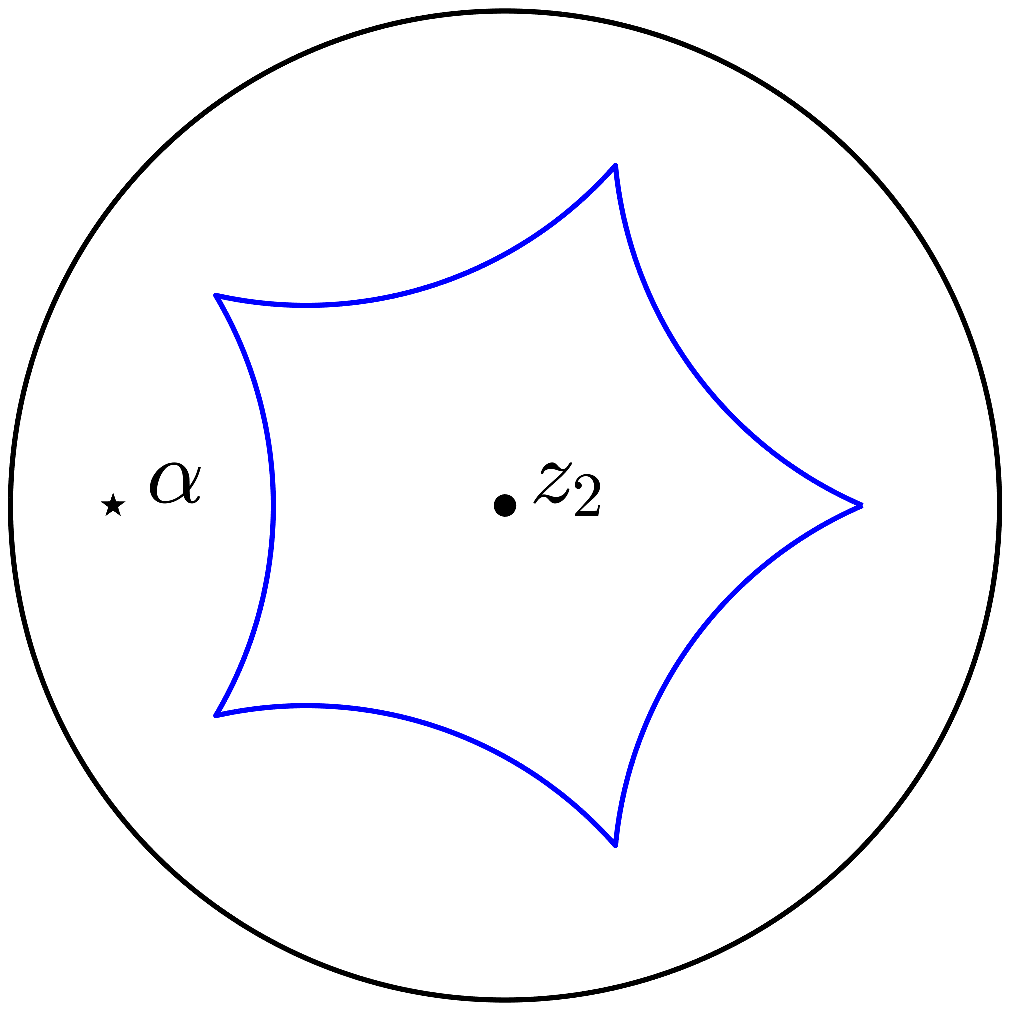}}
\hfill}
\caption{The domain $D$ for computing $\capa(\D,P)$ (left) and the domain $G_0$ for computing $\capa(\D,P_0)$ (right) for $\beta_1=0.5$, $\beta_2=0.6-0.6\i$, $\beta_3=-0.6-0.4\i$, $\beta_4=-0.3+0.6\i$, $\beta_5=0.2+0.5\i$.}
\label{fig:hyp_pol}
\end{figure}

\begin{figure}[hbt] %
\centerline{
\scalebox{0.5}{\includegraphics[trim=0cm 0cm 0cm 0cm,clip]{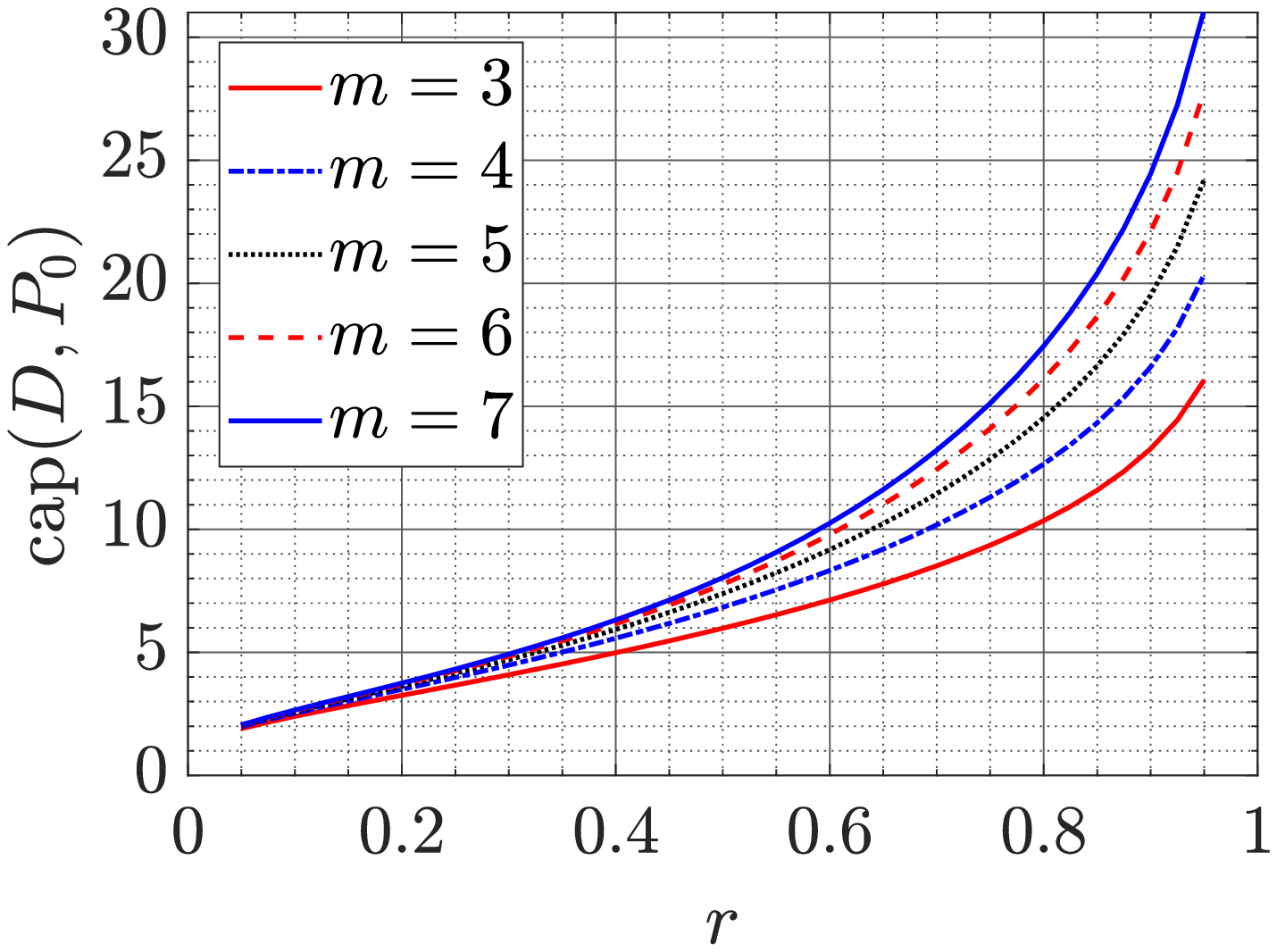}}
}
\caption{The values of $\capa(\D,P_0)$ vs. $r$ for $0.05\le r\le0.95$.}
\label{fig:hyp-pol-r}
\end{figure}

\begin{table}[hbt]
\caption{The values of $\capa(\D,P_0)$ for several values of $m$ and $r$.}
\label{tab:hyp_pol_rm}%
\begin{tabular}{l|l|l|l|l|l}\hline
$r\backslash m$ & $3$ & $4$           & $5$           & $6$           & $7$    \\ \hline
$0.1$ &$2.3993612914$ &$2.5281340146$ &$2.5942462887$ &$2.6324787004$ &$2.6565098093$ \\
$0.2$ &$3.2528141529$ &$3.4946167264$ &$3.6244949167$ &$3.7016779579$ &$3.7510464324$ \\
$0.3$ &$4.0913694284$ &$4.4816138033$ &$4.7030688851$ &$4.8395121547$ &$4.9289424320$ \\
\hline
$0.4$ &$4.9856760383$ &$5.5743497987$ &$5.9308261981$ &$6.1608778512$ &$6.3167666667$ \\
$0.5$ &$5.9799062371$ &$6.8325631892$ &$7.3878902352$ &$7.7670978659$ &$8.0354723585$ \\
$0.6$ &$7.1266240809$ &$8.3279319407$ &$9.1730250087$ &$9.7887982158$ &$10.248675793$ \\
\hline
$0.7$ &$8.5161561610$ &$10.180067164$ &$11.444185389$ &$12.431726677$ &$13.216542846$ \\
$0.8$ &$10.349853454$ &$12.653202360$ &$14.534982854$ &$16.110376899$ &$17.447408861$ \\
$0.9$ &$13.274319210$ &$16.602537086$ &$19.510028327$ &$22.105923933$ &$24.452171599$ \\
\hline %
\end{tabular}
\end{table}


\section{Hyperbolic disks}

In this section we consider disjoint hyperbolic disks $\overline{B}_\rho(x_j,L_j)$ in $\D$ and denote
\[
\overline{B}=\cup_{j=1}^{p} \overline{B}_\rho(x_j,L_j).
\]
By subadditivity of the capacity \cite[Lemma 7.1(3), Thm 9.6]{hkv}, we have
\begin{equation}\label{eq:hyb-cap1}
\capa(\D,\overline{B})\le\sum_{j=1}^p\capa(\D,\overline{B}_\rho(x_j,L_j)).
\end{equation}
We study here whether the set $B$ on the left-hand side of this inequality can be replaced by 
a hyperbolic disk 
under two constraints:
\begin{itemize}
\item[(1)] {\it Isoarea problem}: 
The set $B$ is replaced by a hyperbolic disk $B_\rho(0,L)$ such that the hyperbolic area of $B_\rho(0,L)$ is equal to the sum of the hyperbolic areas of the disks $B_\rho(x_j,L_j)$, i.e.,
\begin{equation}\label{eq:h-area1}
\mbox{h-area}(B_\rho(0,L))=\sum_{j=1}^p \mbox{h-area}(B_\rho(x_j,L_j)).
\end{equation}

\item[(2)] {\it Isoperimetric problem}: 
The set $B$ is replaced by a hyperbolic disk $B_\rho(0,\hat L)$ such that the hyperbolic perimeter of $B_\rho(0,\hat L)$ is equal to the sum of the hyperbolic perimeters of the disks $B_\rho(x_j,L_j)$, i.e.,
\begin{equation}\label{eq:h-preim1}
\mbox{h-perim}(B_\rho(0,\hat L))=\sum_{j=1}^p \mbox{h-perim}(B_\rho(x_j,L_j))
\end{equation}
\end{itemize}

By  \eqref{hypDat0} and \eqref{anncap} 
\begin{equation}\label{eq:hyb-cap-Lj}
\capa(\D,B_\rho(x_j,L_j))=\frac{2\pi}{-\log \th(L_j/2)} \,.
\end{equation}
Similarly, 
\begin{eqnarray}
\label{eq:hyb-cap-L}
\capa(\D,B_\rho(0,L))&=&\frac{2\pi}{-\log \th(L/2)}, \\
\label{eq:hyb-cap-hL}
\capa(\D,B_\rho(0,\hat L))&=&\frac{2\pi}{-\log \th(\hat L/2)} \,.
\end{eqnarray}

By~\cite[Theorem~7.2.2, p. 132]{be}, 
\[
\mbox{h-area}(B_\rho(0,L))=4\pi \,\sh^2\left(\frac{L}{2}\right), \quad
\mbox{h-area}(B_\rho(0,L_j))=4\pi\,\sh^2\left(\frac{L_j}{2}\right),
\]
and 
\[
\mbox{h-perim}(B_\rho(0,\hat L))=2\pi\sh\left(\hat L\right), \quad
\mbox{h-perim}(B_\rho(0,L_j))=2\pi\sh\left(L_j\right),
\]
for $j=1,2,\ldots,p$.
Thus, it follows from~\eqref{eq:h-area1} that the constant $L$ is related to the constants $L_1,L_2,\ldots,L_p$ by 
\begin{equation}\label{eq:hyb-area2}
\sh^2\left(\frac{L}{2}\right)=\sum_{j=1}^p\,\sh^2\left(\frac{L_j}{2}\right),
\end{equation}
Similarly, it follows from~\eqref{eq:h-preim1} that the constant $\hat L$ is related to the constants $L_1,L_2,\ldots,L_p$ by
\begin{equation}\label{eq:h-preim2}
\sh(\hat L)=\sum_{j=1}^p\sh\left(L_j\right).
\end{equation}

\begin{lem}\label{Lem:hL-L}
Let $L$ be defined by~\eqref{eq:hyb-area2} and $\hat L$ be defined by~\eqref{eq:h-preim2}, then $\hat L>L$.
\end{lem}

\begin{proof}
Rewrite the equations~\eqref{eq:hyb-area2} as
\[
\sqrt{\sh^2L+1}-1=\sum_{j=1}^p \left(\sqrt{\sh^2L_j+1}-1\right),
\]
which implies that
\begin{equation}\label{eq:hyb-area2-2}
\sh^2(L)=\left(1+\sum_{j=1}^p \left(\sqrt{\sh^2L_j+1}-1\right) \right)^2-1,
\end{equation}

Define a real function $f(x)$ on the interval $[0,1]$ by
\[
f(x)=\left(x+\sum_{j=1}^p\left(\sqrt{\sh^2L_j+x^2}-x\right)\right)^2-x^2
\]
which can be written as
\begin{equation}\label{eq:hyb-area2-f}
f(x)=2x\,\sum_{j=1}^p\left(\sqrt{\sh^2L_j+x^2}-x\right)+\left(\sum_{j=1}^p\left(\sqrt{\sh^2L_j+x^2}-x\right)\right)^2.
\end{equation}
Then, in view of~\eqref{eq:hyb-area2-2} and~\eqref{eq:h-preim2}, we have $f(0)=\sh^2(\hat L)$ and $f(1)=\sh^2(L)$.

We shall prove that $f(x)$ is decreasing on $[0,1]$. The function $f(x)$ can be written as
\begin{equation}\label{eq:hyb-area2-fg}
f(x)=2x\,g(x)+(g(x))^2
\end{equation}
where the real function $g(x)$ is defined on the interval $[0,1]$ by
\begin{equation}\label{eq:hyb-area2-g}
g(x)=\sum_{j=1}^p\left(\sqrt{\sh^2L_j+x^2}-x\right).
\end{equation}
Hence
\begin{equation}\label{eq:hyb-area2-g2}
g'(x)=\sum_{j=1}^p\left(\frac{x}{\sqrt{\sh^2L_j+x^2}}-1\right)
=-\sum_{j=1}^p\frac{\sqrt{\sh^2L_j+x^2}-x}{\sqrt{\sh^2L_j+x^2}}<0.
\end{equation}
Thus
\begin{eqnarray*}
f'(x) &=& 2g(x)+2xg'(x)+2g(x)g'(x) \\
      &=& 2\sum_{j=1}^p\left(\sqrt{\sh^2L_j+x^2}-x\right)
			   -2x\sum_{j=1}^p\frac{\sqrt{\sh^2L_j+x^2}-x}{\sqrt{\sh^2L_j+x^2}}\\
			& &-2g(x)\sum_{j=1}^p\frac{\sqrt{\sh^2L_j+x^2}-x}{\sqrt{\sh^2L_j+x^2}}\\
      &=& 2\sum_{j=1}^p\left(\sqrt{\sh^2L_j+x^2}-x
			   -x\frac{\sqrt{\sh^2L_j+x^2}-x}{\sqrt{\sh^2L_j+x^2}}
				 -g(x)\frac{\sqrt{\sh^2L_j+x^2}-x}{\sqrt{\sh^2L_j+x^2}}\right),
\end{eqnarray*}
which implies that
\[
f'(x)/2 = \sum_{j=1}^p\left(\left(\sqrt{\sh^2L_j+x^2}-x\right)\left(1-\frac{x}{\sqrt{\sh^2L_j+x^2}}\right)
			-g(x)\frac{\sqrt{\sh^2L_j+x^2}-x}{\sqrt{\sh^2L_j+x^2}}\right)
\]
and further
\begin{equation}\label{eq:hyb-area2-fp}
f'(x)/2=-\sum_{j=1}^p\left(g(x)-\left(\sqrt{\sh^2L_j+x^2}-x\right)\right)\frac{\sqrt{\sh^2L_j+x^2}-x}{\sqrt{\sh^2L_j+x^2}}.
\end{equation}
For all values of $x$ in $(0,1)$, it follows from~\eqref{eq:hyb-area2-g} that $g(x)-\left(\sqrt{\sh^2L_j+x^2}-x\right)>0$ for all indices $j$, and hence~\eqref{eq:hyb-area2-fp} implies that $f'(x)<0$. Thus, $f'(x)$ is decreasing on $[0,1]$ and hence
$\sh^2(\hat L)=f(0)>f(1)=\sh^2(L)$. Since $L,\hat L>0$, the proof follows.
\end{proof}


\begin{thm}\label{thm:hL-Lj}
$\capa(\D,B_\rho(0,\hat L))\le\sum_{j=1}^p\capa(\D,B_\rho(x_j,L_j))$
\end{thm}
\begin{proof}
In view of~\eqref{eq:hyb-cap-Lj} and~\eqref{eq:hyb-cap-hL}, to prove the theorem, we need to prove that
\begin{equation}\label{eq:h-preim4}
\frac{2\pi}{-\log\th(\hat L/2)}\le\sum_{j=1}^p\frac{2\pi}{-\log\th(L_j/2)}\,.
\end{equation}

Note that
\begin{eqnarray*}
\frac{2\pi}{-\log\th(t/2)} 
&=& \frac{2\pi}{-\log\frac{\sh(t/2)}{\ch(t/2)}}
=\frac{2\pi}{-\log\frac{\sh(t)}{\sqrt{\sh^2(t)+1}+1}}
=\frac{2\pi}{\log\frac{\sqrt{\sh^2(t)+1}+1}{\sh(t)}} \\
&=&\frac{2\pi}{\log\left(\sqrt{1+\frac{1}{\sh^2(t)}}+\frac{1}{\sh(t)}\right)}
=\frac{2\pi}{\arsh\left(\frac{1}{\sh(t)}\right)}
\end{eqnarray*}
Thus, equation~\eqref{eq:h-preim4} can be written as
\begin{equation}\label{eq:h-preim5}
\frac{2\pi}{\arsh\left(\frac{1}{\sum_{j=1}^p\sh(L_j)}\right)}
\le\sum_{j=1}^p\frac{2\pi}{\arsh\left(\frac{1}{\sh(L_j)}\right)}.
\end{equation}
Thus, we need to prove that
\begin{equation}\label{eq:h-preim6}
f\left(\sum_{j=1}^p \sh(L_j) \right)
\le \sum_{j=1}^p f\left(\sh(L_j)\right)
\end{equation}
where the function $f\,:\,[0,\infty)\,\to\,[0,\infty)$ is defined by
\[
f(t)=\frac{2\pi}{\arsh\left(\frac{1}{\;t\;}\right)},
\]
for $t\in(0,\infty)$ and
\[
f(0)=\lim_{t\to0^+}f(t)=0.
\]
Since
\[
f'(t) = \frac{2\pi}{t\sqrt{t^2+1}\arsh^2\left(\frac{1}{\;t\;}\right)}>0
\]
for $t\in(0,\infty)$, the function $f(t)$ is strictly increasing.
Thus, by~\cite[7.42(1)]{avv}, to prove~\eqref{eq:h-preim6}, it is enough to show that $f(t)/t$ is deceasing on $(0,\infty)$, which in turn is equivalent to showing that the function
\[
g(t)=\frac{f(1/t)}{1/t}
\]
is increasing on $(0,\infty)$. The function $g(t)$ can be written as
\[
g(t)=\frac{2\pi\,t}{\arsh(t)} = \frac{g_1(t)}{g_2(t)}
\]
where $g_1(t)=2\pi t$ and $g_2(t)=\arsh(t)$. 
Since $g_1(0)=0$, $g_2(0)=0$, and the function
\[
\hat g(t) = \frac{g'_1(t)}{g'_2(t)} = 2\pi\sqrt{1+t^2}
\]
is increasing on $(0,\infty)$, then it follows from~\cite[Theorem~1.25]{avv} that the function $g(t)$ is increasing on $(0,\infty)$. This completes the proof of the theorem.
\end{proof}


\begin{thm}
$\capa(\D,B_\rho(0,L))\le\sum_{j=1}^p\capa(\D,B_\rho(x_j,L_j))$
\end{thm}
\begin{proof}
By Theorem~\ref{thm:hL-Lj}, $L<\hat L$. Then, in view of Lemma~\ref{Lem:hL-L}, we have
\[
\frac{2\pi}{-\log\th(L/2)}\le
\frac{2\pi}{-\log\th(\hat L/2)}\le
\sum_{j=1}^p\frac{2\pi}{-\log\th(L_j/2)}.
\]
Then the proof follows from~\eqref{eq:hyb-cap-Lj} and~\eqref{eq:hyb-cap-L}.
\end{proof}

\section{The dependence of the capacity on the number of vertices}

Let $P_m$ be an equilateral regular hyperbolic polygon with $m$ vertices (see Figure~\ref{fig:Tm-c}). We fix a constant $c>0$, then we consider the sequence $P_m, m=3,4,5,\dots$ with all polygons having the hyperbolic area $c\,.$ 
We consider also the same sequence under the constraint that the perimeters of $P_m$ are equal to $c\,.$ 
For both cases, we show by experiments that the sequence $\capa(\D, P_m)$ is monotone.

\nonsec{\bf Hyperbolic area.} 
Assume $0<c<\pi$ and consider the equilateral regular hyperbolic polygon $P_m$ such that
\[
\mbox{h-area}(P_m)=c, \quad {\rm for \; all}\; m=3,4,\ldots\,,
\]
see Figure~\ref{fig:Tm-c} (left) for $c=3$ and $m=3,5,7$. 

\begin{figure}[hbt] %
\centerline{
\hfill
\scalebox{0.5}{\includegraphics[trim=0 0 0 0,clip]{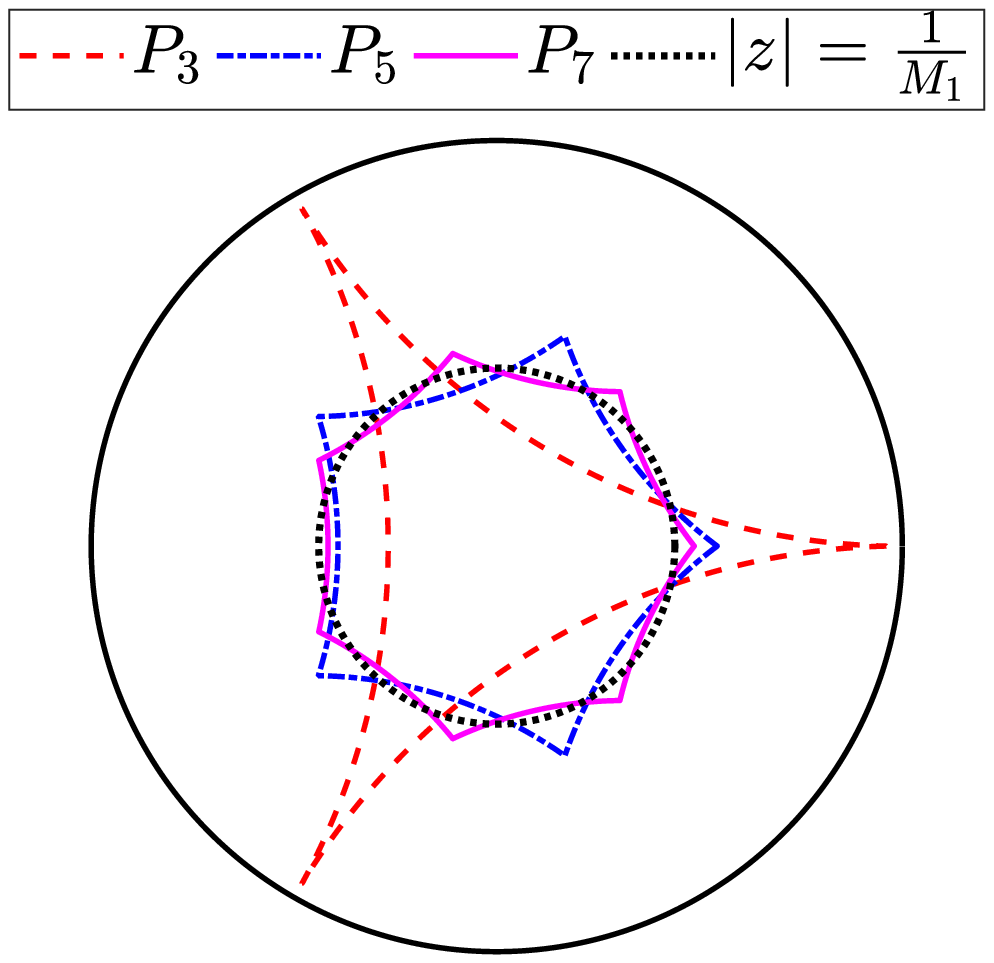}}
\hfill
\scalebox{0.5}{\includegraphics[trim=0 0 0 0,clip]{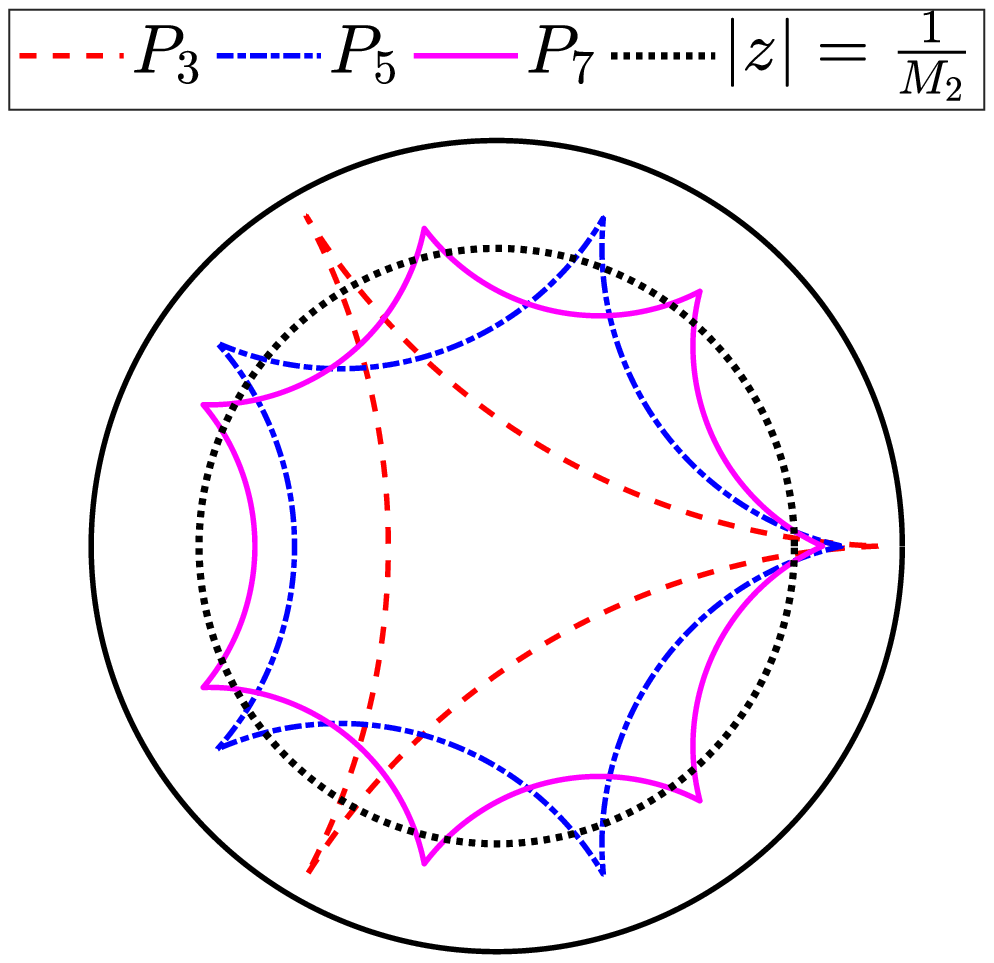}}
\hfill
}
\caption{On the left, the equilateral regular hyperbolic polygons $P_3$, $P_5$, $P_7$ with h-area $c=3$ and the circle $|z|=1/M_1$ where $M_1$ is given by~\eqref{eq:M1}. On the right, the equilateral regular hyperbolic polygons $P_3$, $P_5$, $P_7$ with h-perim $c=20$ and the circle $|z|=1/M_2$ with $M_2$ given by~\eqref{eq:M2}.}
\label{fig:Tm-c}
\end{figure}

Let $M_1$ be chosen such that 
\[
\mbox{h-area}(P_m)=c=\mbox{h-area}(B^2(0,1/M_1)).
\]
Then, by \eqref{hypDat0}, $B^2(0,1/M_1) = B_{\rho}(0,\hat M_1)$ where 
\[
\hat M_1=2\,\arth(1/M_1)\,. 
\]
Hence 
\[
c=\mbox{h-area}(B^2(0,1/M_1))=\mbox{h-area}(B_{\rho}(0,\hat M_1))=4\pi\,\sh^2(\hat M_1/2),
\]
which implies that the constant $\hat M_1$ is given by
\[
\hat M_1=2\,\arsh \sqrt{c/(4\pi)}\,.
\]
Thus, 
\begin{equation}\label{eq:M1}
M_1 = \frac{1}{\th\frac{\hat M_1}{2}}
= \frac{1}{\th(\arsh\sqrt{c/(4\pi)})}
= \sqrt{1+4\pi/c}.
\end{equation}

Note that $M_1>1$ and, by~\eqref{anncap}, 
\[
\capa(\D,B^2(0,1/M_1))=\frac{2\pi}{\log M_1} =  \frac{4\pi}{\log(1+4\pi/c)}. 
\]
We compute numerically the values of $\capa(\D, P_m)$ for several values of $c$ and $m$ using the MATLAB function \verb|annq| with  $n=m\times10^{10}$. The obtained numerical results are presented in Figure~\ref{fig:Tm-LB}. These numerical results lead to the following conjecture.\\

\noindent

\begin{nonsec}{\bf Conjecture.}
Let $0<c<\pi$ and let $P_m$ be a sequence of equilateral regular hyperbolic polygons such that $\mbox{h-area}(P_m)=c$ for all $m=3,4,\ldots\,.$ Then, the sequence $\capa(\D,P_m)$ is decreasing and bounded from below with
\begin{equation}\label{eq:conj-P-area}
\capa(\D,P_m)\ge \frac{2\pi}{\log M_1},
\end{equation}
where $M_1$ is given by~\eqref{eq:M1}.
\end{nonsec}

\begin{figure}[hbt] %
\centerline{
\scalebox{0.45}{\includegraphics[trim=0 0 0 0,clip]{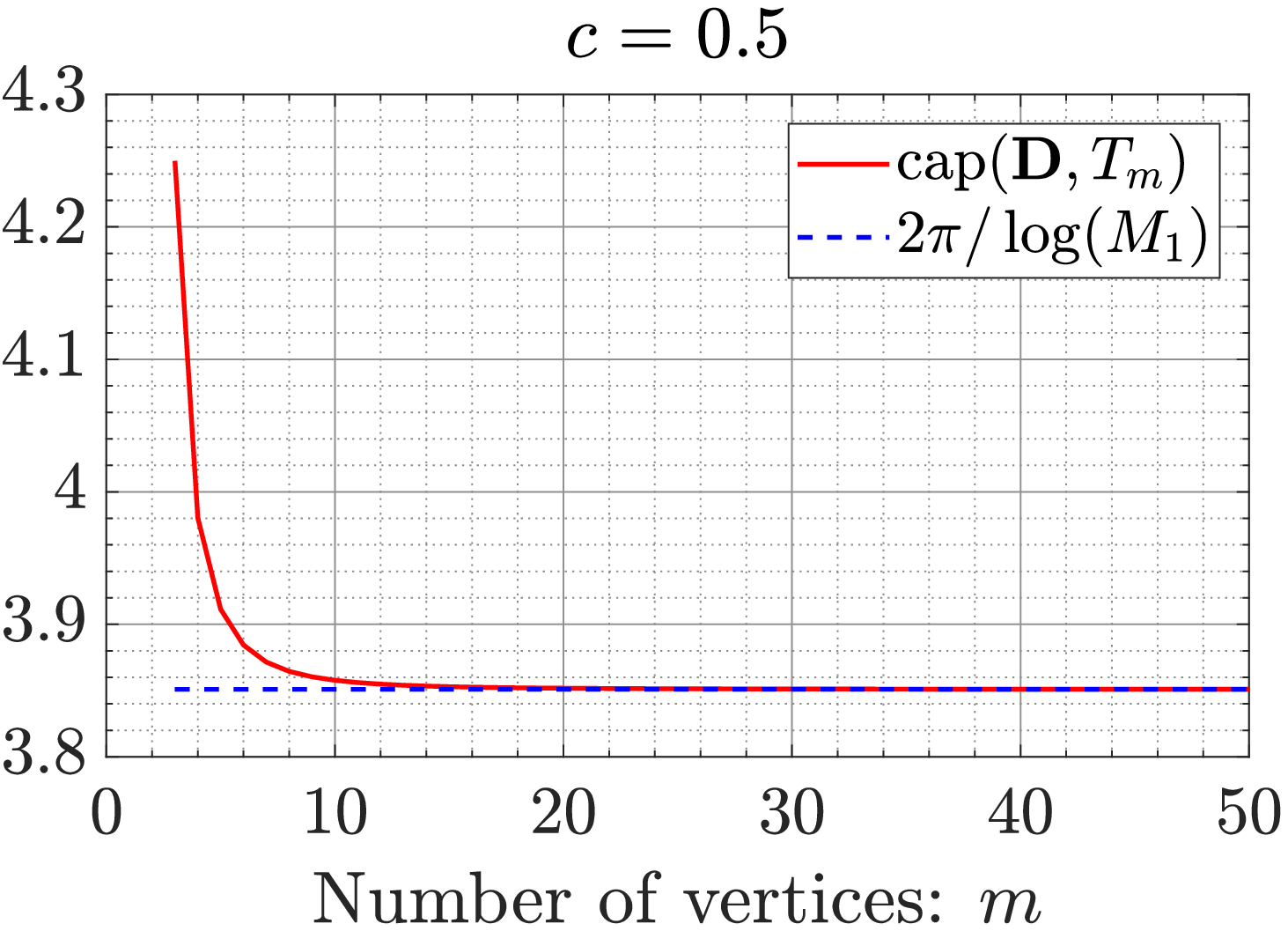}}
\hfill
\scalebox{0.45}{\includegraphics[trim=0 0 0 0,clip]{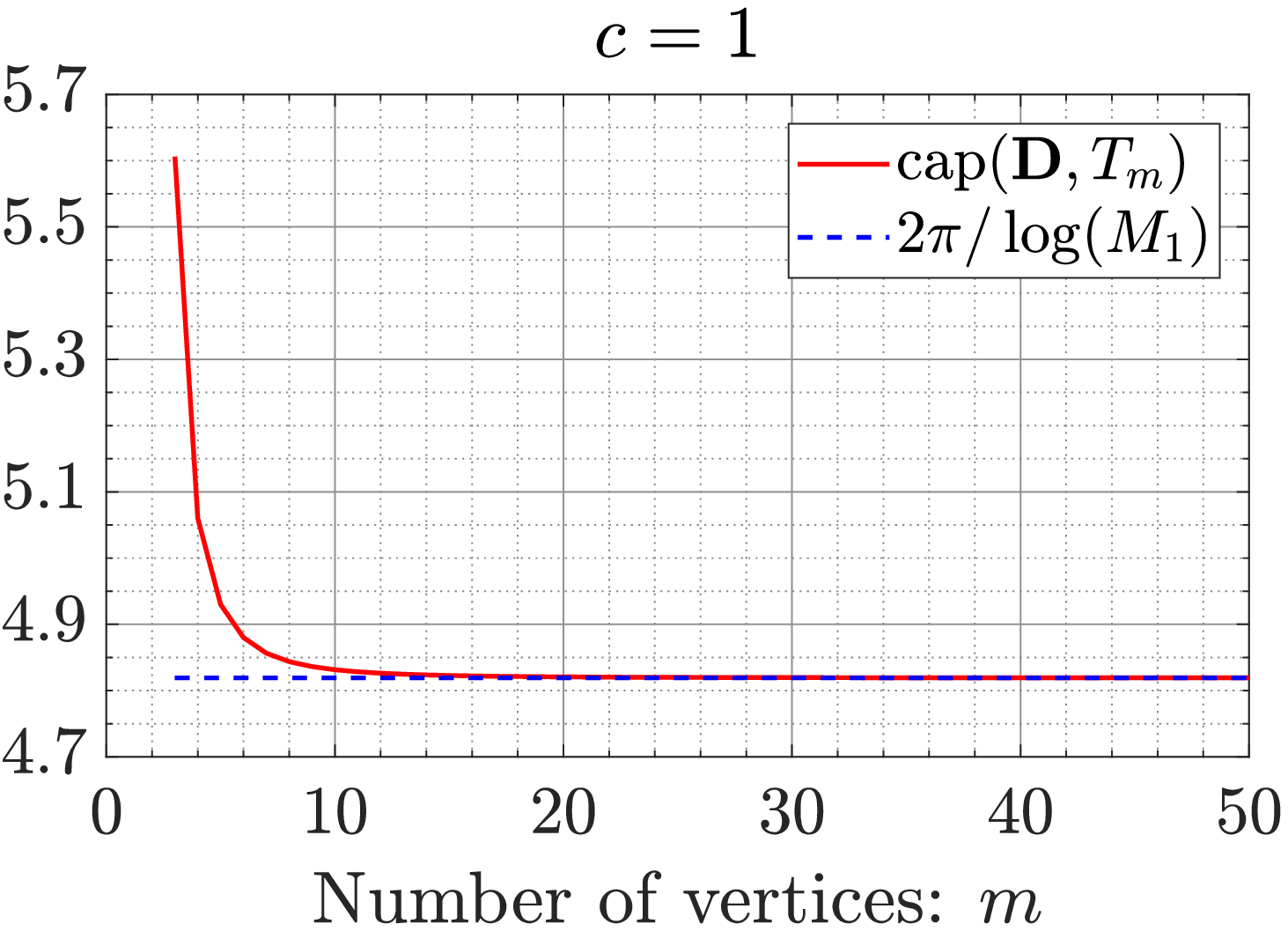}}
}
\vspace{0.3cm}
\centerline{
\scalebox{0.45}{\includegraphics[trim=0 0 0 0,clip]{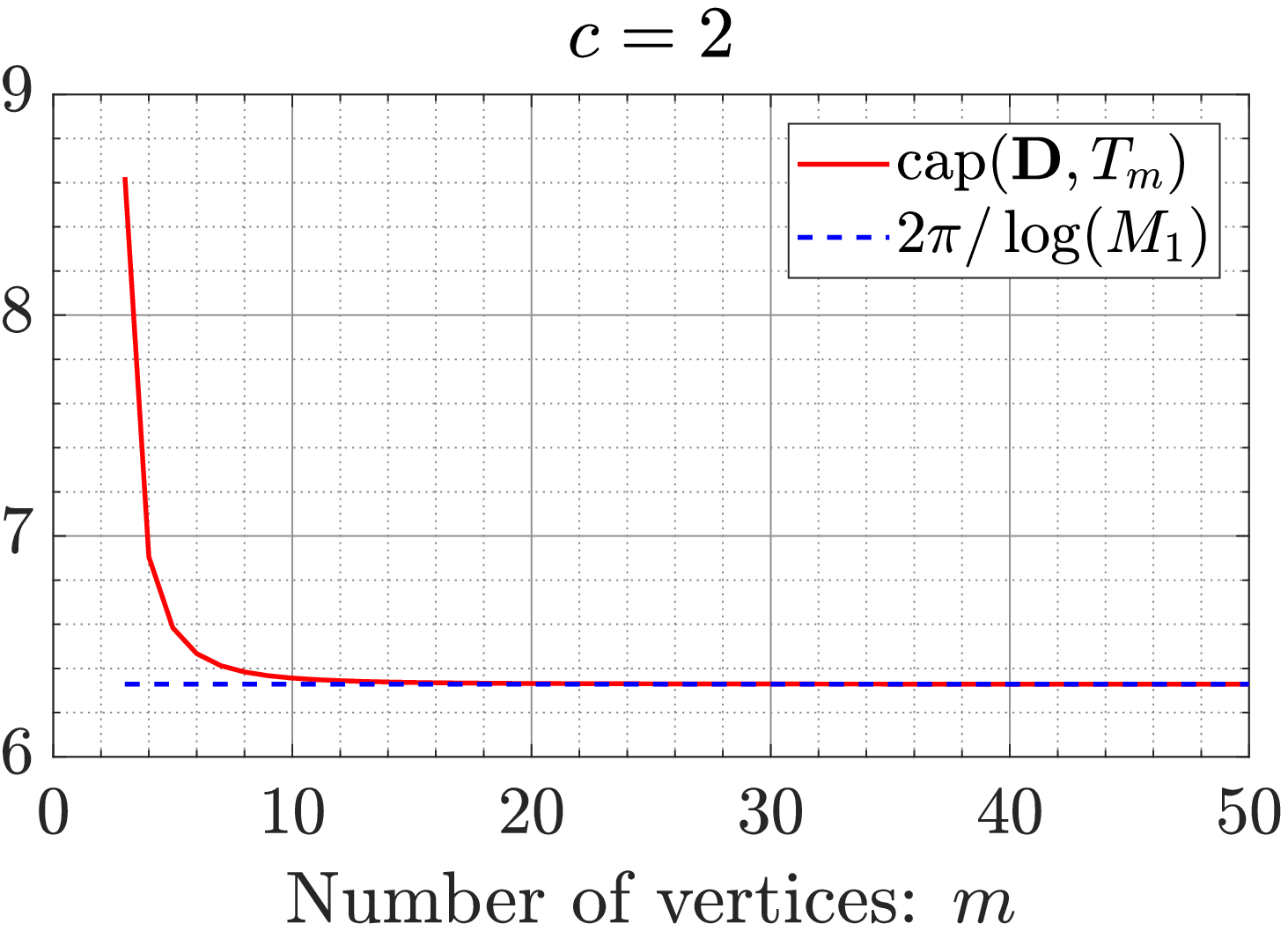}}
\hfill
\scalebox{0.45}{\includegraphics[trim=0 0 0 0,clip]{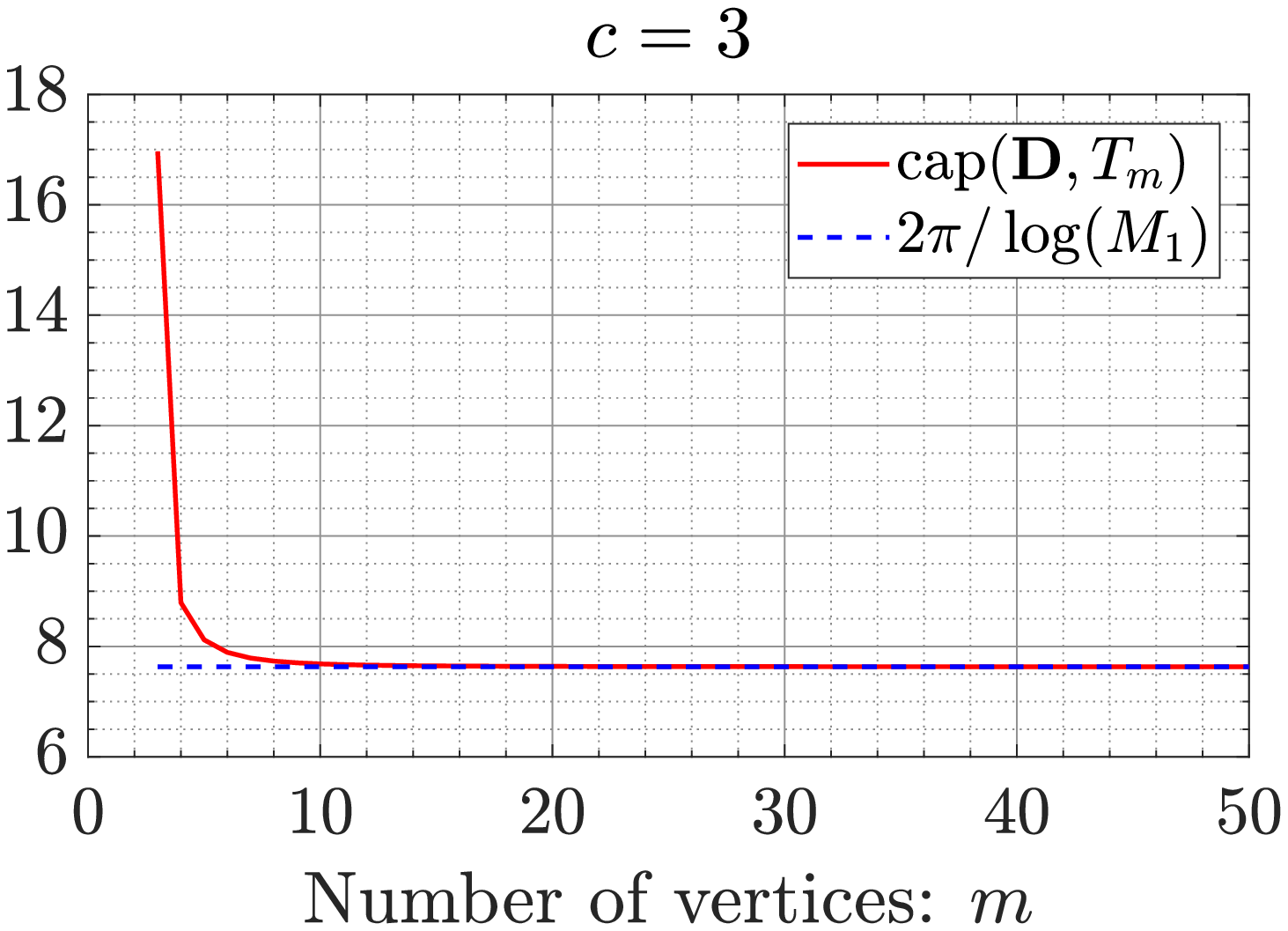}}
}
\caption{The values of $\capa(\D, P_m)$ for several values of $c$ and $m$ where $\mbox{h-area}(P_m)=c$.}
\label{fig:Tm-LB}
\end{figure}

\begin{nonsec}{\bf Hyperbolic perimeter.} 
Assume $c>0$ and consider the equilateral regular hyperbolic polygon $P_m$ such that 
\[
\mbox{h-perim}(P_m)=c, \quad {\rm for \; all}\; m=3,4,\ldots\,,
\]
see Figure~\ref{fig:Tm-c} (right) for $c=20$ and $m=3,5,7$.

Let $M_2$ be chosen such that 
\[
\mbox{h-perim}(P_m)=c=\mbox{h-perim}(S^1(0,1/M_2)),
\]
where $S^1(0,1/M_2)=\{z:|z|=1/M_2\}$. Then, by \eqref{hypDat0}, $S^1(0,1/M_2)=\{z:\rho(0,z)=\hat M_2\}$ where
\[
\hat M_2=2\arth (1/M_2).
\]
Thus 
\[
c=\mbox{h-perim}(S^1(0,1/M_2))
=\mbox{h-perim}(\{z:\rho(0,z)=\hat M_2\})=2\pi\,\sh(\hat M_2),
\]
which implies that
\[
\hat M_2=\arsh\frac{c}{2\pi},
\]
and hence
\begin{equation}\label{eq:M2}
M_2 = \frac{1}{\th\frac{\hat M_2}{2}}
= \frac{1}{\th\left(\frac{1}{2} \arsh\frac{c}{2\pi}\right)}
= \sqrt{1+\frac{4\pi^2}{c^2}}+\frac{2\pi}{c}.
\end{equation}

The MATLAB function \verb|annq| with  $n=m\times10^{10}$ is used to compute numerically the capacities $\capa(\D, P_m)$ for several values of $c$ and $m$. The obtained results are presented in Figure~\ref{fig:Tm-UB} and suggest the following conjecture, where by~\eqref{anncap}, 
\[
\capa(\D,B^2(0,1/M_2))=\frac{2\pi}{\log M_2}. 
\]
\end{nonsec}

\noindent
\begin{nonsec}{\bf Conjecture.}
Let $c>0$ and let $P_m$ be a sequence of equilateral regular hyperbolic polygons such that $\mbox{h-perim}(P_m)=c$ for all $m=3,4,\ldots\,.$ Then, the sequence $\capa(\D,P_m)$ is increasing and bounded from above with
\begin{equation}\label{eq:conj-P-perim}
\capa(\D,P_m)\le \frac{2\pi}{\log M_2},
\end{equation}
where $M_2$ is given by~\eqref{eq:M2}.

\begin{rem}
At the stage when the manuscript was processed for publication, we have learned about F.W. Gehring's work~\cite[Corollary~6]{g} which implies that the upper bound
\eqref{eq:conj-P-perim} here holds for all sets $E \subset \D$ with the hyperbolic perimeter at most $c$. 
Note that in~\cite{g} the hyperbolic metric differs by factor $(1/2)$ from our hyperbolic metric and therefore~\cite[Corollary~6]{g} gives
the constant $\pi$ in the formula for $M_2$ whereas we have $2 \pi$.
\end{rem}

\end{nonsec}

\begin{figure}[hbt] %
\centerline{
\scalebox{0.45}{\includegraphics[trim=0 0 0 0,clip]{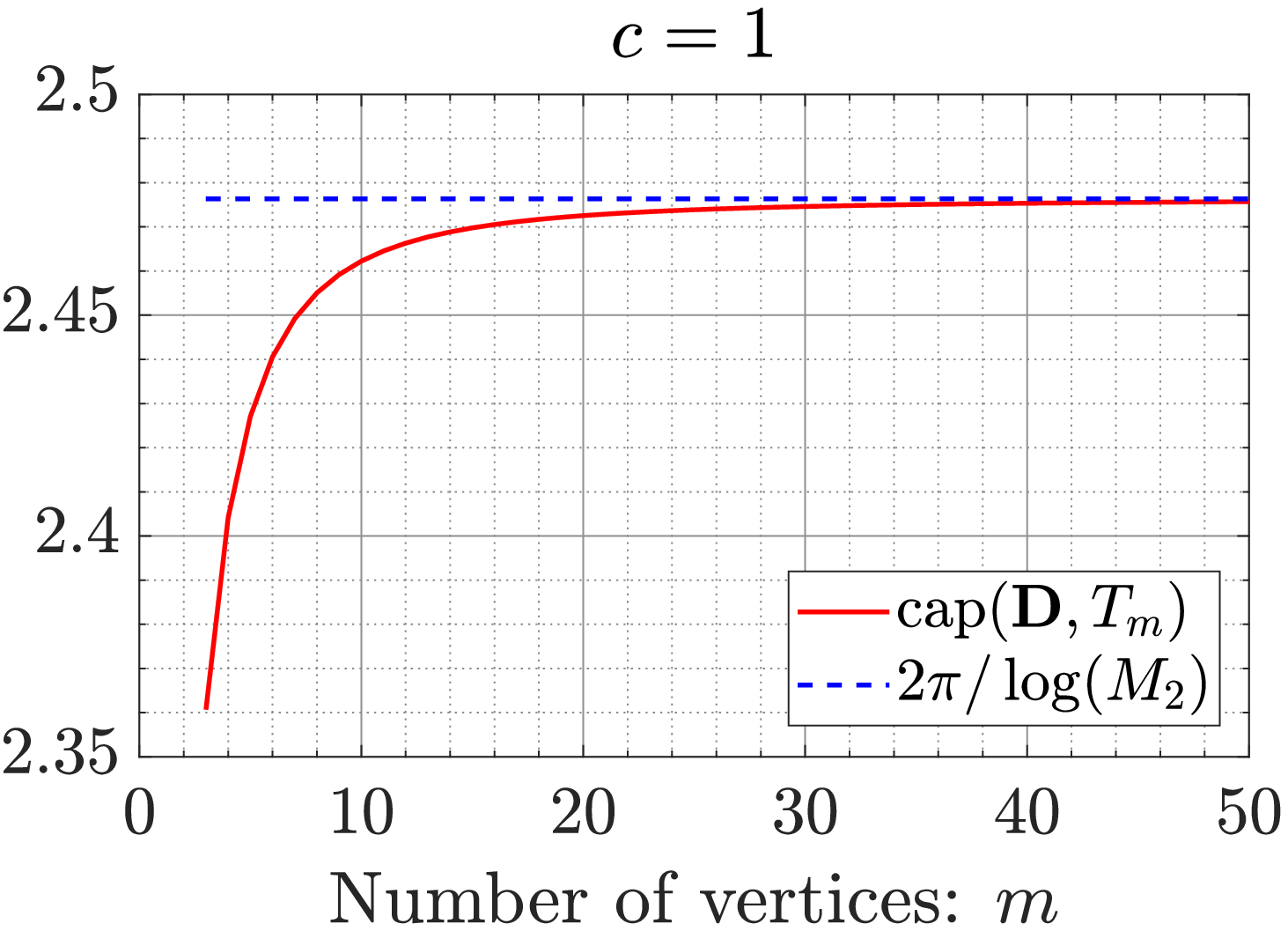}}
\hfill
\scalebox{0.45}{\includegraphics[trim=0 0 0 0,clip]{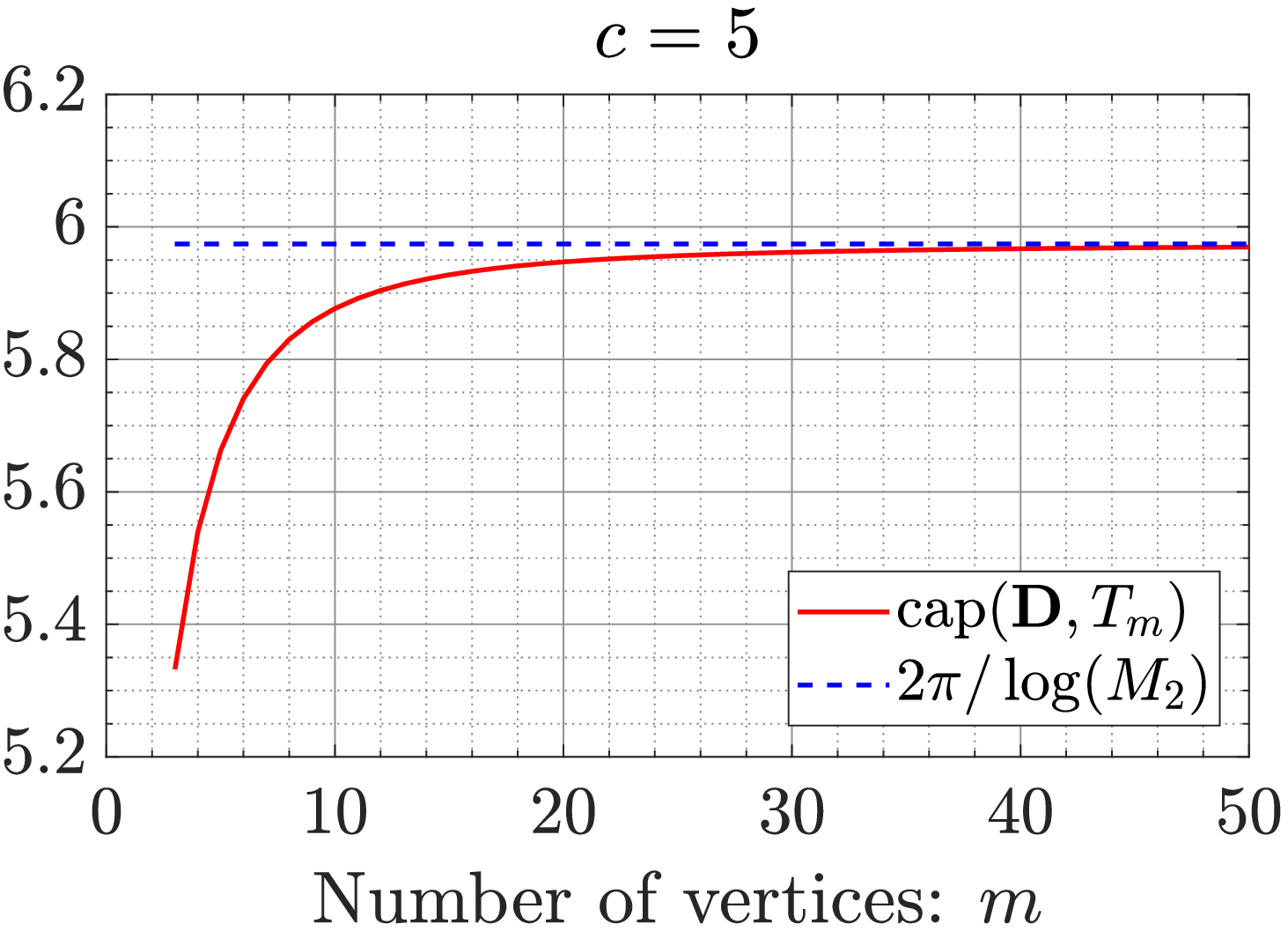}}
}
\vspace{0.3cm}
\centerline{
\scalebox{0.45}{\includegraphics[trim=0 0 0 0,clip]{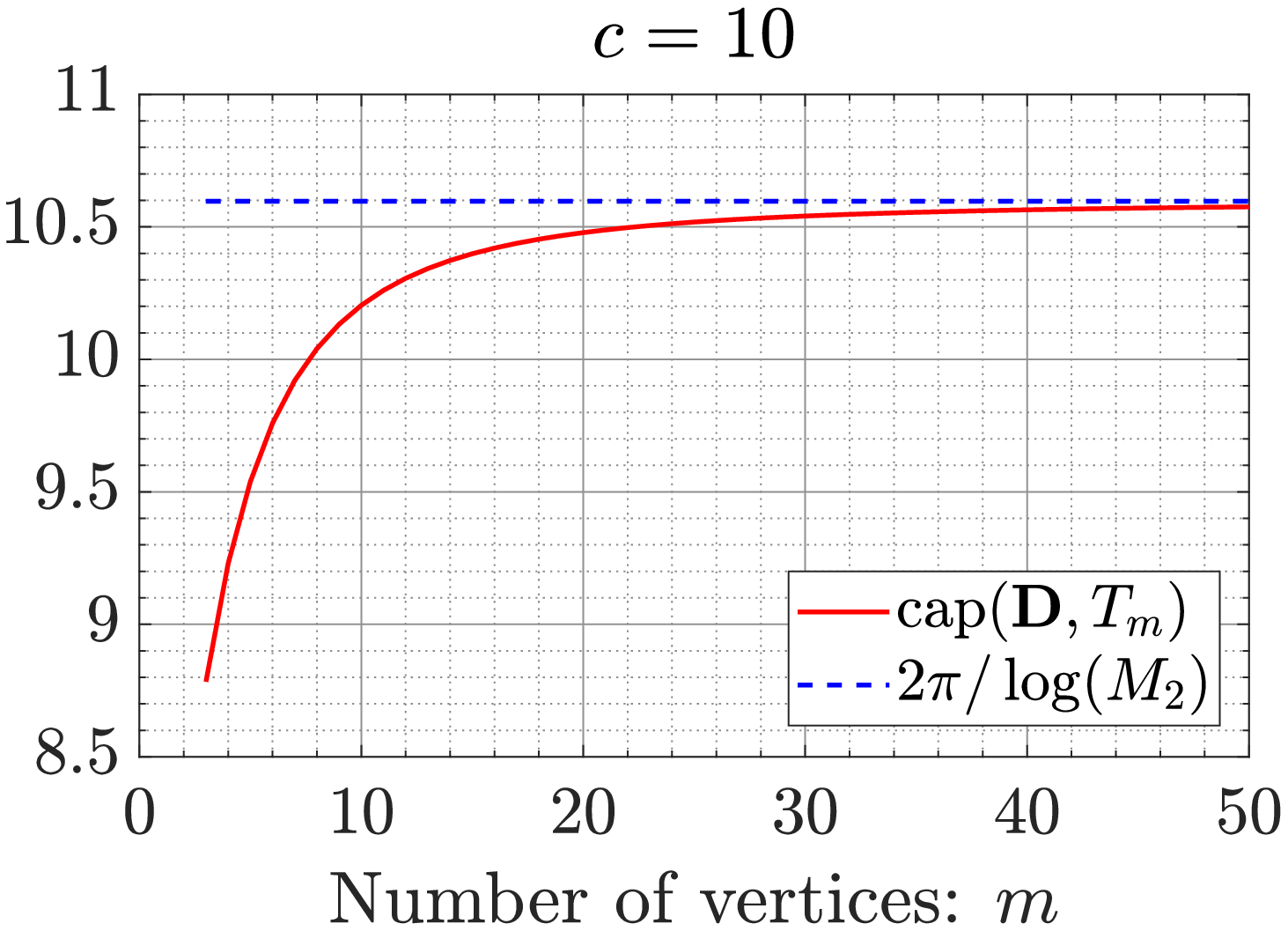}}
\hfill
\scalebox{0.45}{\includegraphics[trim=0 0 0 0,clip]{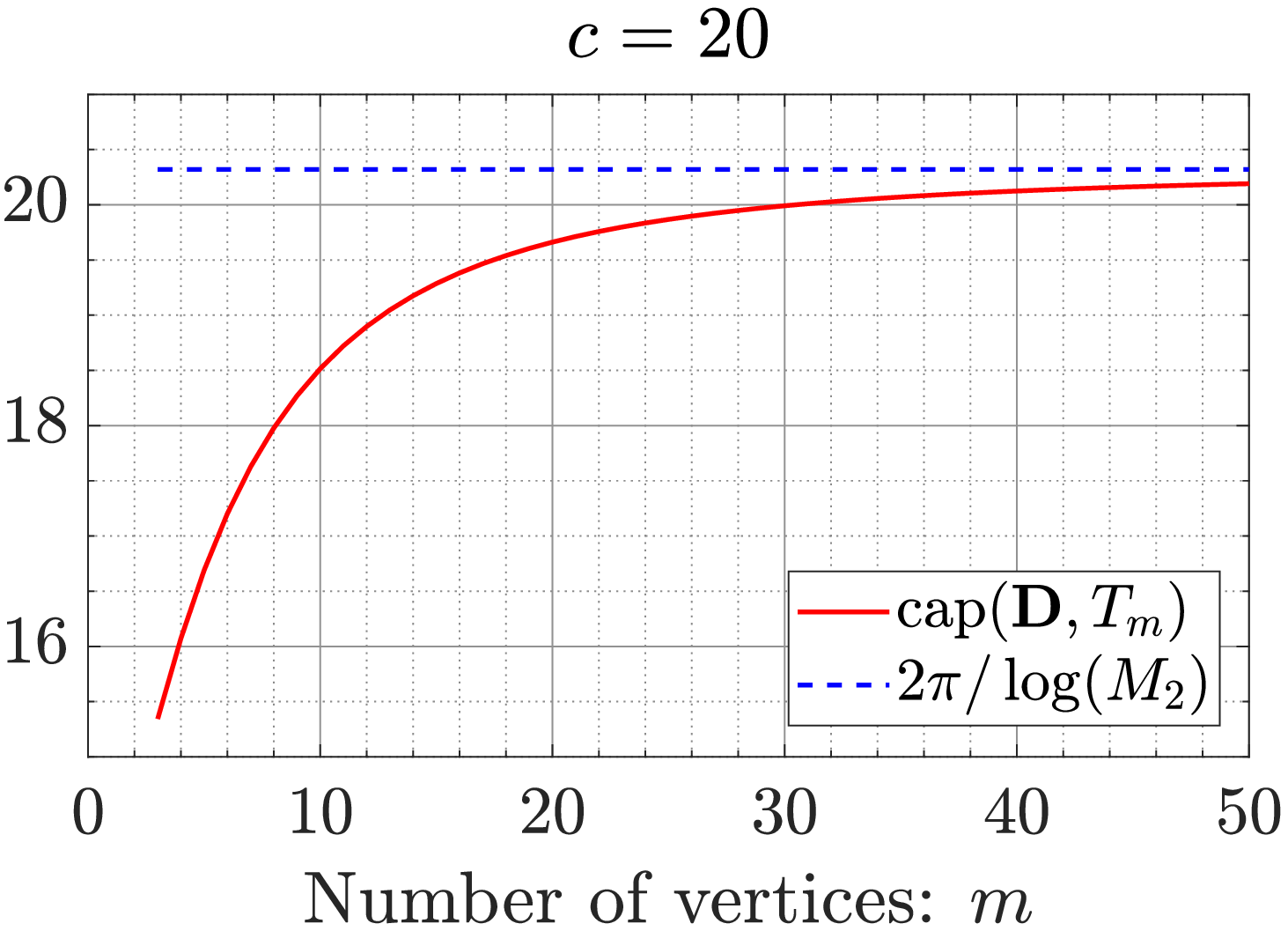}}
}
\caption{The values of $\capa(\D, P_m)$ for several values of $c$ and $m$ where $\mbox{h-perim}(P_m)=c$.}
\label{fig:Tm-UB}
\end{figure}

\section{Subadditivity and additivity of the modulus}

For curve families $\Gamma_j$, $j=1,2,\ldots$, we have the following 
subadditivity property by Lemma \ref{le71},
\cite[Lemma 7.1(3)]{hkv}
\begin{equation}\label{eq:sub-add}
\mM\left(\bigcup_{i=1}^{\infty}\Gamma_j\right) \le
\sum_{j=1}^\infty \mM\left(\Gamma_j\right),
\end{equation}
and if the families $\Gamma_j$ are separate we have the additivity  Lemma \ref{le73}, \cite[Lemma 7.3]{hkv}
\begin{equation}\label{eq:add}
\mM\left(\bigcup_{i=1}^{\infty}\Gamma_j\right) =
\sum_{j=1}^\infty \mM\left(\Gamma_j\right),
\end{equation}

 \nonsec{\rm {\bf Equilateral hyperbolic triangle.} 
In this section, for a given $s\in(0,1)$, we consider the equilateral hyperbolic triangle $T$ with vertices
\[
a_0=s, \quad a_1=s\,e^{\i\theta}, \quad a_2=s\,e^{2\i\theta}, \quad \theta=\frac{2\pi}{3}.
\]
Assume that the three angles of this triangle are equal to $\beta$ and the length of all three sides of the triangle is $b$.
Let $D$ be the doubly connected domain between the unit circle and the triangle $T$. Then, there exists a unique real constant $q\in(0,1)$ and a unique conformal mapping $f:D\to \Omega=B^2(0,1) \setminus \overline{B^2(0,q)}$ normalized by $f(\alpha)>0$ where $\alpha\in(s,1)$ (see~Figure~\ref{fig:tri-equ}).
Let $D_j$ be the sub-domain of $D$ defined by
\[
D_j=\left\{z\,:\, z\in G, \quad \theta_{j-1}<\arg z<\theta_j\right\}, \quad j=1,2,3,
\]
where $\theta_j=j\theta$, for $j=0,1,2,3$.
By symmetry, the conformal mapping $w=f(z)$ will map each sub-domain $D_j$ onto a sub-domain $\Omega_j$ of $\Omega$ where
\[
\Omega_j=\left\{w\,:\, q<|w|<1, \quad \theta_{j-1}<\arg w<\theta_{j}\right\}, \quad j=1,2,3.
\]
The sub-domains $D_1$ and $\Omega_1$ are the shaded domains in Figure~\ref{fig:tri-equ}. 

Let
\begin{equation}\label{eq:Gam-1}
\Gamma=\Delta(S^1(q) , S^1 ; B^2(0,1) \setminus  B^2(0,q))
\end{equation}
be the family of all curves joining $S^1(q)$ and $S^1$ in $B^2(0,1) \setminus  B^2(0,q)$. Let also 
\[
\Delta=\{u\,s\;:\:u\in S^1, \;\;q\le s\le 1\}\subset\Gamma.
\]
In this case $\Delta$ consists of radial segments joining the boundary components of the ring $B^2(0,1) \setminus  B^2(0,q)$. Moreover, choose the separate subfamilies 
\[
\Delta_j=\{u\,s\;:\:u\in S^1, \;\; \theta_{j-1}\le\arg(u)\le\theta_j, \;\;q\le s\le 1\}, \quad j=1,2,3.
\]
Then by \eqref{eq:add} 
\[
\mM(\Gamma)=\mM(\Delta)=\sum_{j=1}^{3}\mM(\Delta_j).
\]
For the above conformal mapping $f$, we see that $f^{-1}(S^1(q))=T$. Hence, by~\cite[Lemma 7.1(3)]{hkv} and~\cite[Theorem 2. 7]{jen},
\begin{equation}\label{eq:M-finv-1}
\mM(\Delta(T , S^1 ; \D)) = \sum_{j=1}^{m}\mM(f^{-1}\Delta_j)
\end{equation}
where $f^{-1}\Delta_j=\{f^{-1}\circ\gamma\;:\;\gamma\in\Delta_j\}$. By symmetry and by~\cite[5.17]{vu88}, we have
\[
\mM(f^{-1}\Delta_3)=\mM(f^{-1}\Delta_2)=\mM(f^{-1}\Delta_1)=\mM(\Delta_1).
\]
Consequently, 
\begin{equation}\label{eq:cap-B2T}
\capa(\D,T) = \mM(\Delta(T,S^1;\D))=\sum_{j=1}^3\mM(f^{-1}\Delta_j)=3\mM(f^{-1}\Delta_1)
=3\mM(\Delta_1).
\end{equation}

\begin{figure}[hbt] %
\centerline{
\hfill
\scalebox{0.45}{\includegraphics[trim=0 0 0 0,clip]{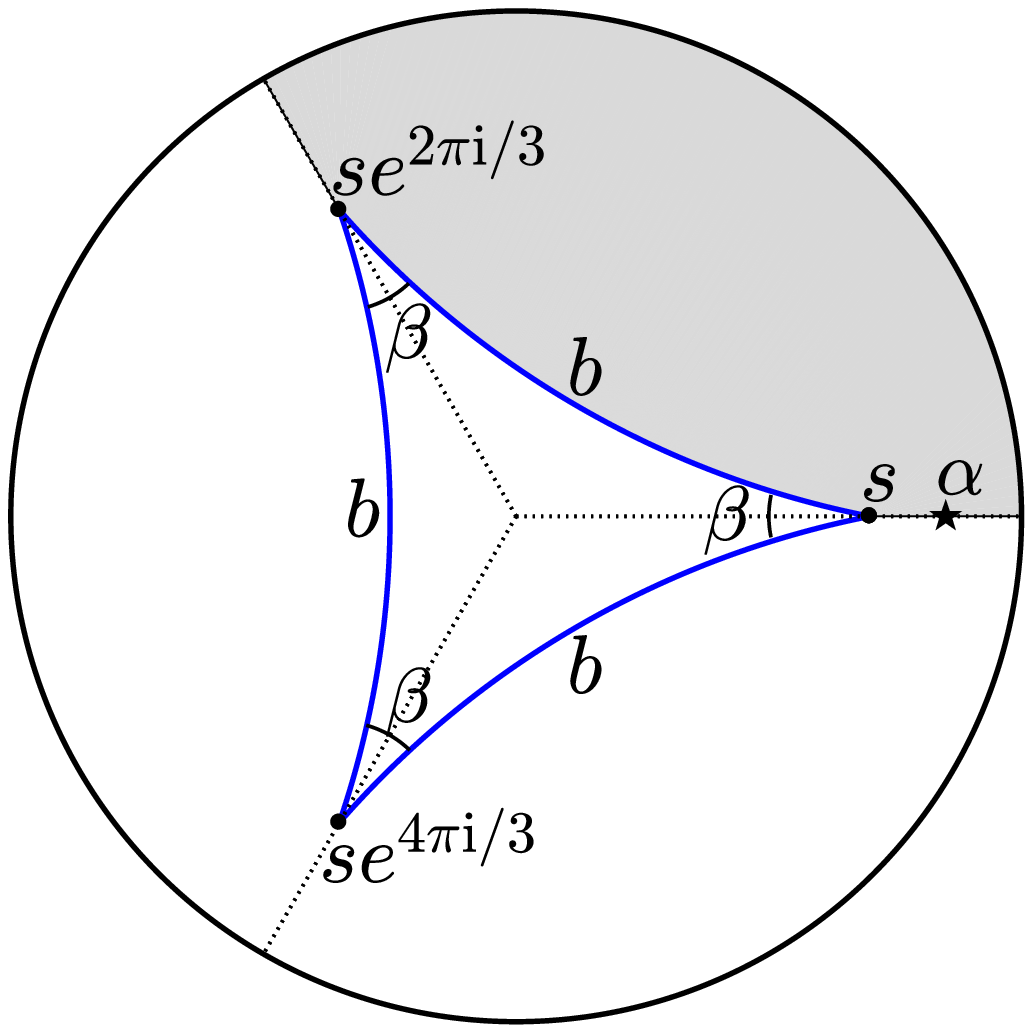}}
\hfill
\scalebox{0.45}{\includegraphics[trim=0 0 0 0,clip]{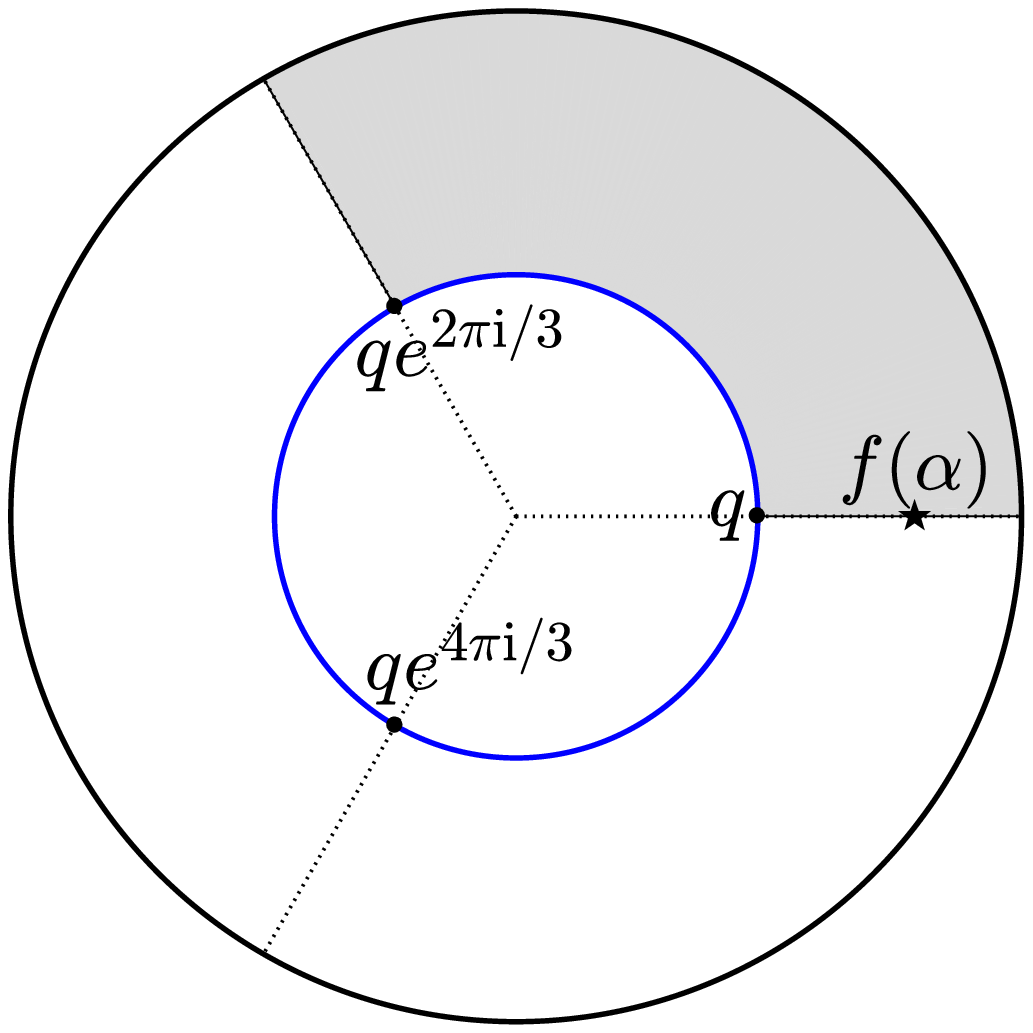}}
\hfill
}
\caption{The doubly connected domain $D$ and its image $\Omega$.}
\label{fig:tri-equ}
\end{figure}

\nonsec{\rm {\bf Lower bound for the capacity of equilateral hyperbolic triangle.}}
Let $\hat{T}$ be the connected set consisting of the three segments $re^{\i k\alpha}$ for $\alpha=2\pi/3$, $0\le r\le s$ and $k=0,1,2$. Let also $\hat{D}$ be the domain obtained by removing $\hat{T}$ from the unit disk $\D$ (see Figure~\ref{fig:tri-equ-D}).

\begin{figure}[hbt] %
\centerline{
\hfill
\scalebox{0.45}{\includegraphics[trim=0 0 0 0,clip]{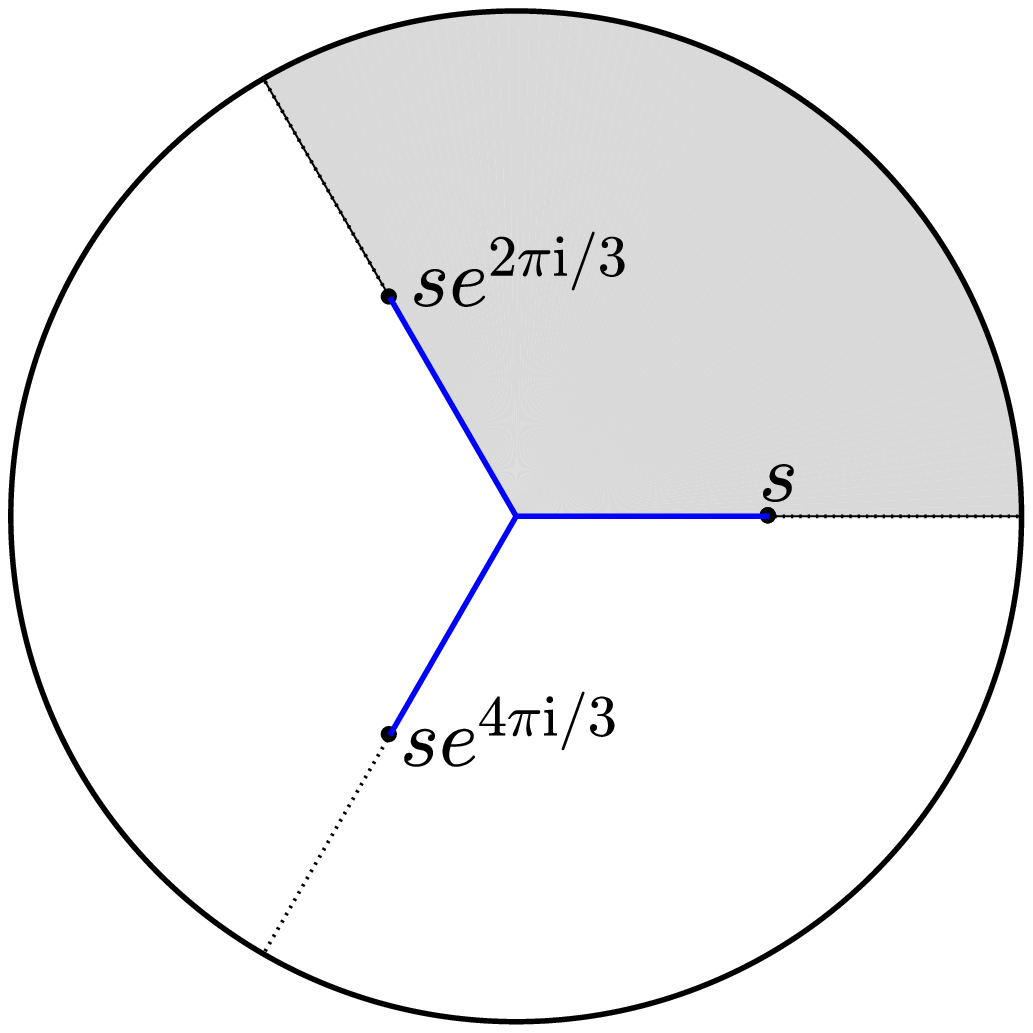}}
\hfill
\scalebox{0.45}{\includegraphics[trim=0 0 0 0,clip]{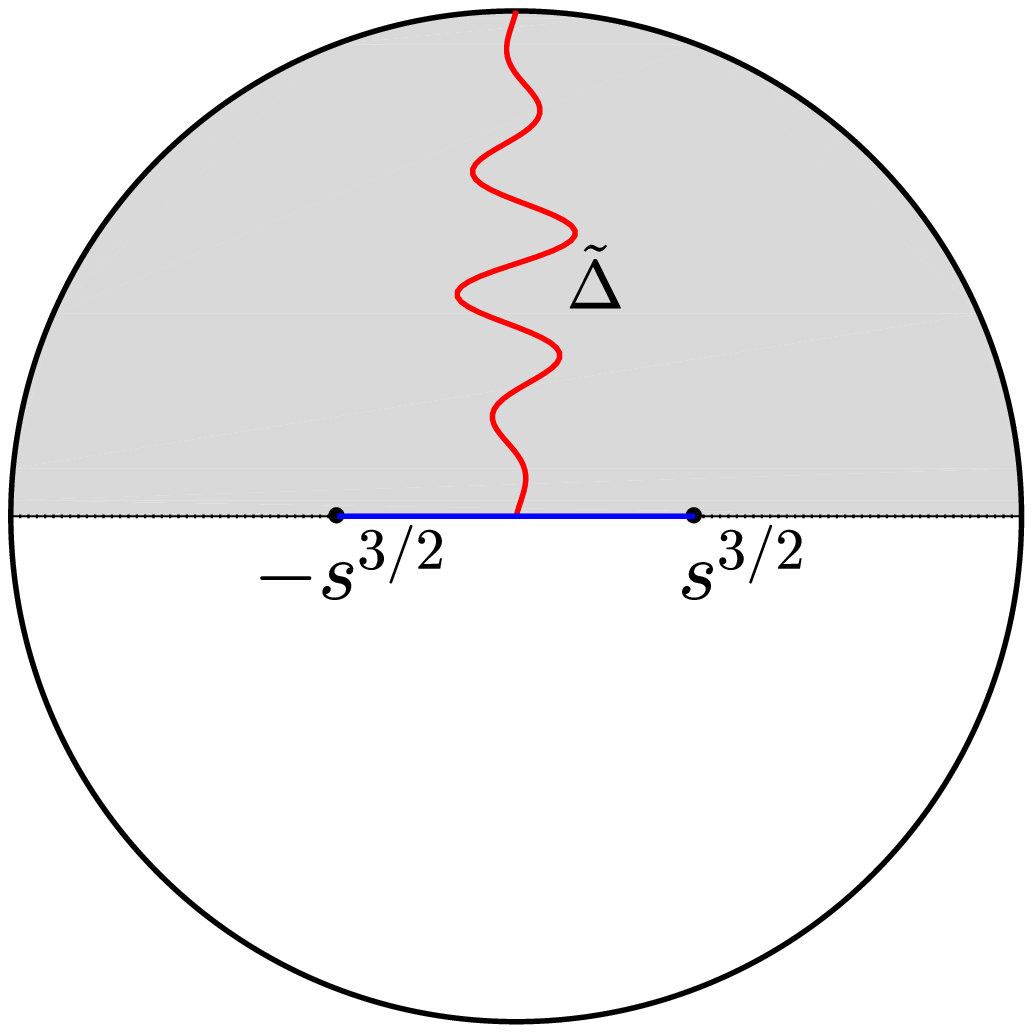}}
\hfill
}
\caption{The doubly connected domain $\hat D$ (left) and the family of curves $\tilde\Delta$.}
\label{fig:tri-equ-D}
\end{figure}

\begin{lem}\label{Lem:LB-Tri-h} 
The capacity $\capa(\D,\hat{T})$ is given by
\[
\capa(\D,\hat{T}) = \frac{6\pi}{\mu\left(s^3\right)}.
\]
\end{lem}
\begin{proof}
Let $\hat{D}_1$ be the sub-domain of $\hat{D}$,
\[
\hat{D}_1=\left\{z\,:\, z\in \hat{G}, \quad \theta_0<\arg z<\theta_1\right\},
\]
(see the shaded domain in Figure~\ref{fig:tri-equ-D} (left)). 
The domain $\hat{D}_1$ can be mapped by the conformal mapping 
\[
g(z)= z^{3/2}
\]
which  onto the upper half of the unit disk and the two segments from $s$ to $0$ and from $0$ to $s\,e^{\theta_1\i}$ are mapped onto the segment $[-s^{3/2},s^{3/2}]$. 
Let $\tilde\Delta$ be the family of curves in the upper half of the unit disk connecting the segment $[-s^{3/2},s^{3/2}]$ to the upper half of the unit circle (see Figure~\ref{fig:tri-equ-D} (right)). 
Then using the same argument as above,
\begin{equation}\label{eq:hT-1}
\capa(\D,\hat{T})=\mM(\Delta(\hat{T},S^1;\D))=3\mM(\tilde{\Delta}).
\end{equation}

By symmetry, it follows from~\cite[5.20,~7.32]{vu88},
\begin{equation}\label{eq:hT-2}
\mM(\tilde\Delta)=\frac{1}{2}\gamma_2\left(\frac{1}{\th\frac{1}{2}\rho_{\D}(-s^{3/2},s^{3/2})}\right)= \frac{\pi}{\mu\left(\th\frac{1}{2}\rho_{\D}(-s^{3/2},s^{3/2})\right)}.
\end{equation}
Using~\eqref{myrho}, we have
\[
\th\frac{1}{2}\rho_{\D}(-s^{3/2},s^{3/2})
=\frac{2s^{3/2}}{s^3+1}.
\]
and hence
\begin{equation}\label{eq:hT-3}
\mM(\tilde\Delta) = \frac{\pi}{\mu\left(2s^{3/2}/(s^3+1)\right)}.
\end{equation}
Using the formula (\cite[(5.4)]{avv})
\[
\mu(r)=\frac{1}{2}\mu\left(\left(\frac{r}{1+\sqrt{1-r^2}}\right)^2\right),
\]
with $r=2s^{3/2}/(s^3+1)$ and hence $\left(\frac{r}{1+\sqrt{1-r^2}}\right)^2=s^3$, we obtain
\begin{equation}\label{eq:hT-4}
\mM(\tilde\Delta) = \frac{2\pi}{\mu\left(s^3\right)}.
\end{equation}
The proof then follows from~\eqref{eq:hT-1} and~\eqref{eq:hT-4}.
\end{proof}

\begin{lem}\label{Lem:LB-Tri} 
The capacity $\capa(\D,T)$ can be estimated by
\begin{equation}\label{eq:ineq-L}
\capa(\D,T)\ge \frac{6\pi}{\mu\left(s^3\right)}.
\end{equation}
\end{lem}
\begin{proof}
Let $\overline{\Delta}=\Delta(T , S^1 ; G)$ and $\hat{\Delta}=\Delta(\hat{T} , S^1 ; \hat{G})$, then $\overline{\Delta}<\hat{\Delta}$. Hence, by Lemma~\ref{le72},
\[
\capa(\D,T)=\mM(\overline{\Delta})\ge \mM(\hat{\Delta})=\capa(\D,\hat{T}),
\]
and the proof follows from Lemma~\ref{Lem:LB-Tri-h}.
\end{proof}

\nonsec{\rm {\bf An upper bound for the capacity of equilateral hyperbolic triangle.}}

\begin{lem}\label{Lem:UB-Tri} 
The capacity $\capa(\D,T)$ can be estimated by
\begin{equation}\label{eq:ineq-U}
\capa(\D,T)\le \frac{3\pi}{\mu\left(\sqrt{3}s/\sqrt{s^4+s^2+1}\right)}.
\end{equation}
\end{lem}
\begin{proof}
For $j=1,2,3$, let $\hat\Delta_j=\Delta([a_{j-1},a_j],S^1,\D)$ where $[a_{j-1},a_j]$ is the hyperbolic segment from $a_{j-1}$ to $a_j$ (see Figure~\ref{fig:tri-equ}).
Here $a_j=s\,e^{j\i\theta}$ for $j=0,1,2,3$ and $\theta=2\pi/3$, i.e., $a_3=a_0$.
Then~\cite[7.32]{vu88}, 
\begin{equation}\label{eq:cap-D}
\mM(\hat\Delta_j)= \gamma_2\left(\frac{1}{\th\frac{1}{2}\rho_{\D}(a_{j-1},a_j)}\right)
= \frac{2\pi}{\mu\left(\th\frac{1}{2}\rho_{\D}(a_{j-1},a_j)\right)}, \quad j=1,2,3.
\end{equation}
Note that $\mM(\hat\Delta_1)=\mM(\hat\Delta_2)=\mM(\hat\Delta_3)$. Let also $\hat\Delta^\ast_j$ be the set of all curves in $\hat\Delta_j$ not intersecting the hyperbolic line through $a_{j-1}$ and $a_j$, then by symmetry,
\begin{equation}\label{eq:cap-D*}
\mM(\hat\Delta^\ast_j)
= \frac{\pi}{\mu\left(\th\frac{1}{2}\rho_{\D}(a_{j-1},a_j)\right)}, \quad j=1,2,3.
\end{equation}
By~\eqref{myrho}, we have
\begin{equation}\label{eq:rho-s}
\th\frac{\rho_{\D}(a_{0},a_1)}{2}
=\th\frac{\rho_{\D}(s,se^{\i\alpha})}{2}
=\frac{\sqrt{3}s}{\sqrt{s^4+s^2+1}},
\end{equation}
and hence
\begin{equation}\label{eq:cap-D*1}
\mM(\hat\Delta^\ast_1)
=\frac{\pi}{\mu\left(\th\frac{1}{2}\rho_{\D}(a_{0},a_1)\right)}
=\frac{\pi}{\mu\left(\frac{\sqrt{3}s}{\sqrt{s^4+s^2+1}}\right)}.
\end{equation}
It follows from~\eqref{eq:cap-D*} that $\mM(\hat\Delta^\ast_1)=\mM(\hat\Delta^\ast_2)=\mM(\hat\Delta^\ast_3)$ since $T$ is an equilateral hyperbolic triangle. 
Thus, for the family of curves $\overline{\Delta}=\Delta(T , S^1 ; G)$, by the subadditivity, Lemma \ref{le71}, we have
\begin{equation}\label{eq:cap-Db}
\capa(\D,T) = \mM(\overline{\Delta})\le 3\mM(\hat\Delta^\ast_1).
\end{equation}
The proof follows from~\eqref{eq:cap-D*1}. 
\end{proof}

Lemmas~\ref{Lem:LB-Tri} and~\ref{Lem:UB-Tri} yield the following corollary. 

\begin{cor}\label{cor:UB-Tri} 
For the equilateral hyperbolic triangle $T$ with the vertices $s$, $se^{2\pi\i/3}$, and $se^{4\pi\i/3}$, we have
\begin{equation}\label{eq:ineq-LU}
\frac{6\pi}{\mu\left(s^3\right)} \le
\capa(\D,T)
\le \frac{3\pi}{\mu\left(\sqrt{3}s/\sqrt{s^4+s^2+1}\right)}.
\end{equation}
\end{cor}

The inequality~\eqref{eq:ineq-LU} is confirmed by numerical results as in Figure~\ref{fig:h_tri_cap_LU} where the numerical values of $\capa(\D,T)$ are computed using the MATLAB function \verb|annq| with $n=3\times 2^{12}$.

\begin{figure}[hbt] %
\centerline{
\scalebox{0.6}{\includegraphics[trim=0 0 0 0,clip]{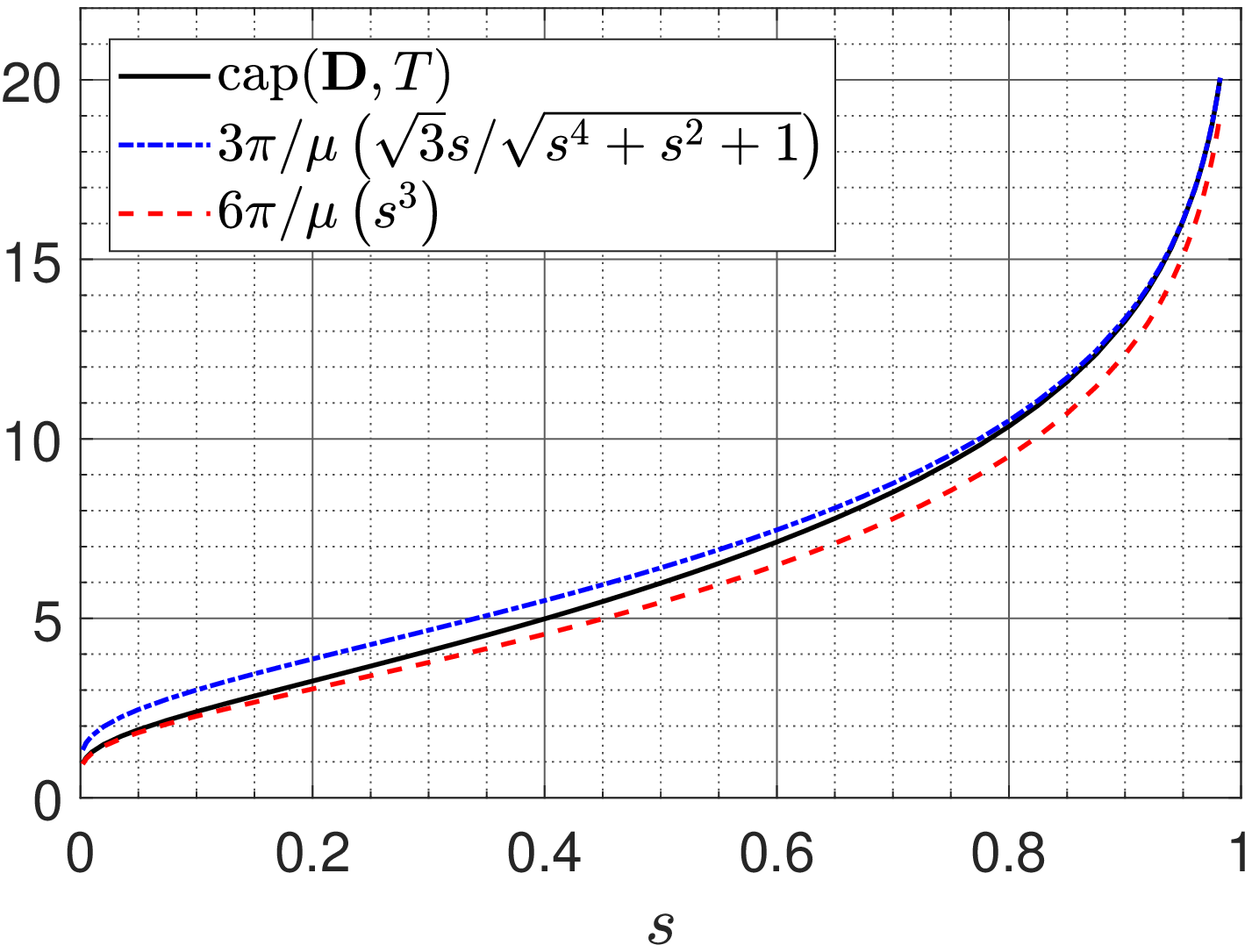}}
}
\caption{The capacity $\capa(\D,T)$ of the equilateral hyperbolic triangle $T$ with the vertices $s$, $se^{2\pi\i/3}$, and $se^{4\pi\i/3}$.}
\label{fig:h_tri_cap_LU}
\end{figure}

\nonsec{\rm {\bf The capacity, the hyperbolic perimeter, and the hyperbolic area.}
For the equilateral hyperbolic triangle $T$ (shown in Figure~\ref{fig:tri-equ}), we have (\cite[p.~150]{be})
\begin{equation}\label{eq:pre-are-1}
2\ch\frac{b}{2}\sin\frac{\beta}{2}=1.
\end{equation}
Let $u$ be the hyperbolic perimeter and $v$ be the hyperbolic area of the triangle $T$. Then $u=\mbox{h-perim}(T)=3b$ and $v=\mbox{h-area}(T)=\pi-3\beta$. Thus, $b=u/3$, $\beta=(\pi-v)/3$, and equation~\eqref{eq:pre-are-1} can be written as
\begin{equation}\label{eq:pre-are-3}
2\ch\frac{u}{6}\sin\frac{\pi-v}{6}=1.
\end{equation}

The upped bound of $\capa(\D,T)$ given in~\eqref{eq:ineq-LU} can be written in terms of the hyperbolic perimeter $u$ of the triangle $T$.
Since $T$ is an equilateral hyperbolic triangle, then $u=3\rho_{\D}(a_{0},a_1)$, and~\eqref{eq:rho-s} implies that
\begin{equation}\label{eq:thu-s}
\th\frac{u}{6}=\frac{\sqrt{3}s}{\sqrt{s^4+s^2+1}}.
\end{equation}
Then it follows from~\eqref{eq:ineq-U} that
\begin{equation}\label{eq:cap-Db-2}
\capa(\D,T) \le \frac{3\pi}{\mu(\th\frac{u}{6})}
\end{equation}

The upped bound also can be written in terms of the hyperbolic area $v$ of the triangle $T$.
Using the formula (\cite[(5.2)]{avv})
\[
\mu(r)\mu\left(\sqrt{1-r^2}\right)=\frac{\pi^2}{4},
\]
with $r=\th(u/6)$ and hence $\sqrt{1-r^2}=1/\ch(u/6)$, we obtain
\begin{equation}\label{eq:pre-are-4}
\mu\left(\th\frac{u}{6}\right)\mu\left(\frac{1}{\ch\frac{u}{6}}\right)=\frac{\pi^2}{4},
\end{equation}
which in view of~\eqref{eq:pre-are-3} can be written as
\begin{equation}\label{eq:pre-are-5}
\mu\left(\th\frac{u}{6}\right)\mu\left(2\sin\frac{\pi-v}{6}\right)=\frac{\pi^2}{4}.
\end{equation}
Hence, the inequality~\eqref{eq:cap-Db-2} can be written in terms of the hyperbolic  area $v$ of the triangle $T$ as 
\begin{equation}\label{eq:conj-cap-A2T}
\capa(\D,T) \le \frac{12}{\pi}\mu\left(2\sin\frac{\pi-v}{6}\right).
\end{equation}

It is possible to write also the lower bound of $\capa(\D,T)$ given in~\eqref{eq:ineq-LU} in terms of the hyperbolic perimeter $u$ and the hyperbolic area $v$ of the triangle $T$. This can be done by writing $s^3$ in terms of $u$ and $v$. 
By~\eqref{eq:thu-s}, we can show that
\[
\sh\frac{u}{6}=\frac{\sqrt{3}\,s}{1-s^2}, 
\]
and hence
\[
\sh\frac{u}{6}\th^2\frac{u}{6}=\frac{3\sqrt{3}\,s^3}{1-s^6}.
\]
Consequently, we have
\begin{equation}\label{eq:s^3-u}
s^3=-\sigma+\sqrt{\sigma^2+1}, \quad \sigma=\frac{3\sqrt{3}}{2\sh\frac{u}{6}\th^2\frac{u}{6}},
\end{equation}
which rewrite~\eqref{eq:ineq-LU} in terms of the hyperbolic perimeter $u$.
Further, in view of~\eqref{eq:thu-s} and~\eqref{eq:pre-are-3}, we can show that
\[
\ch\frac{u}{6}=\frac{\sqrt{s^4+s^2+1}}{1-s^2}, \quad
\sin\frac{\pi-v}{6}=\frac{1-s^2}{2\sqrt{s^4+s^2+1}},
\]
and hence
\[
\tan\frac{\pi-v}{6}=\frac{1-s^2}{\sqrt{3}(1+s^2)}.
\]
Thus
\begin{equation}\label{eq:s^3-v}
s^3=\left(\frac{1-\tau}{1+\tau}\right)^{3/2}, \quad \tau=\sqrt{3}\tan\frac{\pi-v}{6},
\end{equation}
rewrite~\eqref{eq:ineq-LU} using the hyperbolic area $v$.

\bigskip

\noindent\textbf{Acknowledgments.}

The authors are indebted to Prof. Alex Solynin for bringing \cite{so1,so2,soz} to our attention.
We are also grateful to Prof. D. Betsakos for informing us  about F.W. Gehring's result~\cite[Corollary~6]{g}.

\end{document}